\documentclass[amscd,amssymb,verbatim,10pt]{amsart}
\usepackage[fontsize = 10bp]{fontsize}

\usepackage{verbatim}
\usepackage{amssymb, latexsym, amsmath, amscd, array, graphicx}
\usepackage{hyperref}
\usepackage[all]{xy}
\usepackage[scr]{rsfso}

\usepackage{mathtools}
\usepackage{color}
\hypersetup{
    colorlinks=true,
    linkcolor=red,
    urlcolor=black,
    linktoc=all,
    citecolor=blue
           }
\usepackage{tikz-cd}
\usepackage{bbm,graphicx}

\theoremstyle{plain}

\newtheorem{thm}{Theorem}[section]
\newtheorem{theorem}[thm]{Theorem}

\newtheorem{lemma}[thm]{Lemma}

\newtheorem{proposition}[thm]{Proposition}
\newtheorem{corollary}[thm]{Corollary}

\theoremstyle{definition}
\newtheorem{definition}[thm]{Definition}

\newtheorem{remark}[thm]{Remark}

\newtheorem{example}[thm]{Example}

\DeclareMathOperator{\cat}{\mathsf{cat}}
\DeclareMathOperator{\TC}{\mathsf{TC}}
\DeclareMathOperator{\dcat}{\mathsf{dcat}}
\DeclareMathOperator{\dTC}{\mathsf{dTC}}
\DeclareMathOperator{\icat}{\mathsf{icat}}
\DeclareMathOperator{\iTC}{\mathsf{iTC}}
\DeclareMathOperator{\proj}{proj}

\DeclareMathOperator{\Ker}{{\rm Ker}}

\DeclareMathOperator{\supp}{{\rm supp}}

\def\exp{\protect\operatorname{exp}}

\def\B{{\mathcal B}}

\def\H{{\mathcal H}}
\def\U{{\mathscr U}}
\def\V{{\mathscr V}}
\def\W{{\mathscr W}}

\def\F{{\mathbb F}}

\newcommand \pa[2]{\frac{\partial #1}{\partial #2}}

\def\scr{\mathcal}

\def\B{{\scr B}}

\def\C{{\mathbb C}}
\def\Z{{\mathbb Z}}
\def\Q{{\mathbb Q}}
\def\R{{\mathbb R}}

\def\1{\hbox{\rm\rlap {1}\hskip.03in{\rom I}}}
\def\Bbbone{{\rm1\mathchoice{\kern-0.25em}{\kern-0.25em}
{\kern-0.2em}{\kern-0.2em}I}}

\def\pa{\partial}

\def\wt{\widetilde}
\def\wh{\widehat}

\def\ov{\overline}

\long\def\forget#1\forgotten{} %

\newcommand\ver[1]{\marginpar{\tiny Changed in Ver \VER}}

\date{\today}

\begin{document}

\title[Intertwining category and complexity]{Intertwining category and complexity}

\author[Ekansh Jauhari]{Ekansh Jauhari}

\address{Ekansh Jauhari, Department of Mathematics, University
of Florida, 358 Little Hall, Gainesville, FL 32611-8105, USA.}
\email{ekanshjauhari@ufl.edu}



\begin{abstract}
We develop the theory of the intertwining distributional versions of the LS-category and the sequential topological complexities of a space $X$, denoted by $\icat(X)$ and $\iTC_m(X)$, respectively. We prove that they satisfy most of the nice properties as their respective distributional counterparts $\dcat(X)$ and $\dTC_m(X)$, and their classical counterparts $\cat(X)$ and $\TC_m(X)$, such as homotopy invariance and special behavior on topological groups. We show that the notions of $\iTC_m$ and $\dTC_m$ are different for each $m \ge 2$ by proving that $\iTC_m(\H)=1$ for all $m \ge 2$ for Higman's group $\H$. Using cohomological lower bounds, we also provide various examples of locally finite CW complexes $X$ for which $\icat(X) > 1$, $\iTC_m(X) > 1$, $\icat(X) = \dcat(X) = \cat(X)$, and $\iTC(X) = \dTC(X) = \TC(X)$.
\end{abstract}




\maketitle

\section{Introduction}\label{Introduction}
In the field of robotics, the ultimate aim is to develop autonomously functioning mechanical systems that can understand well-defined descriptions of tasks and execute them without any further human intervention. One of the simplest such tasks is navigating the system through a sequence of positions inside a given configuration space $X$. Constructing a robot that can accomplish this
task autonomously is sometimes called the \textit{motion planning problem.} 

We consider the following sequential motion planning problem. Given a system with configuration space $X$, a number $m \ge 2$, and positions $x_1, x_2, \ldots, x_m$ in $X$, find a continuous motion planning algorithm that takes as input the ordered $m$-tuple of positions $\ov{x}=(x_1,x_2, \ldots, x_m) \in X^m$ and produces as output the continuous motion of the system from $x_1$ to $x_m$ via the $m-2$ intermediate positions $x_2, \ldots, x_{m-1}$ attained in that order.

M. Farber observed in~\cite{Far1} that even when $m = 2$ and the system just needs to move from an initial position to a final position in $X$, such a continuous motion planning algorithm does not exist on all of $X \times X$ unless $X$ is contractible. So, he developed in~\cite{Far1},~\cite{Far2} the notion of \textit{topological complexity} of $X$, denoted $\TC(X)$, which is one less than the minimum number of subspaces into which $X \times X$ needs to be partitioned such that a continuous (partial) motion planning algorithm can exist on each of the subspaces. In a sense, $\TC(X)$ provides the minimum number of rules required to plan the above motion, and thus, gives a measure of the complexity in planning the motion. For any general $m \ge 2$, the idea of topological complexity was extended by Y. B. Rudyak in~\cite{Rud} in a natural way to get the $m^{th}$ \textit{sequential topological complexity} of $X$, denoted $\TC_m(X)$, which measures the complexity of planning the sequential motion.

Ideally, one would want to minimize the complexity of motion planning. Given a configuration space $X$, this can be achieved for some advanced robotic systems with special features. In a recent work with A. Dranishnikov~\cite{DJ}, we considered systems that can break into finitely many weighted pieces at the initial position $x$ so that all the pieces travel \textit{independently} to the desired final position $y$ where they can reassemble back into the system. The continuous motion can now be planned on $X \times X$ using the unordered collection of the system's weighted pieces instead of using the system as a whole. We let $n$ be the maximum number of pieces that the system can break into while traveling between all possible pairs of positions $(x,y)$. To measure the complexity of such a motion, the notion of the \textit{distributional topological complexity} of $X$, denoted $\dTC(X)$, was introduced in~\cite{DJ} as one less than the minimum value of $n$ for which the system can achieve such a motion by traveling as smaller pieces. Whenever $X$ is non-contractible, the system must break into at least two pieces for a continuous motion planning algorithm to exist.

If we consider repeated breaking and reassembling of such a system in the same weighted pieces, then the sequential motions can be planned. This natural extension of $\dTC$ was obtained by us recently in~\cite{Ja}, where the generalized notion of the $m^{th}$ \textit{sequential distributional topological complexity} of $X$, denoted $\dTC_m(X)$, was 
studied. Around the same time, B. Knudsen and S. Weinberger independently developed in~\cite{KW} their probabilistic version of $\TC_m$, called the $m^{th}$ \textit{sequential analog topological complexity}, denoted $A\TC_m(X)$ for a space $X$.

In~\cite{Far2},~\cite{Far3}, Farber had seen $\TC$ as a measure of the minimal level of randomness in motion planning algorithms. To measure this randomness, he considered ordered probability distributions. This is equivalent to the breaking of the system into an \textit{ordered} collection of pieces as opposed to the unordered collection of pieces considered in the definition of $\dTC_m$. Thus, $\dTC_m(X)\le\TC_m(X)$. Similarly, since~\cite{KW} argued that their notions measure analog randomness by considering unordered probability distributions, it follows that $A\TC_m(X)\le\TC_m(X)$. Even though there is a formal difference in the definition of $\dTC_m$ and $A\TC_m$ (see~\cite[Section 1]{Dr} and Section~\ref{delicate}), it was conjectured in~\cite[Conjecture 1.2]{KW} that the two notions coincide on metric spaces. In that direction, we showed in~\cite[Section 8]{Ja} that $\dTC_m(X) \le A\TC_m(X)$. So, given an advanced system with a configuration space $X$, the notion of $\dTC_m(X)$ seems to give the most promising lower bound so far to the complexity of the motion planning algorithms on $X$. 

\subsection{Intertwining motion}\label{intmot}
A further improvement to $\dTC$ was proposed in the Epilogue of our earlier work~\cite{DJ}. The basic idea is again of breaking and reassembling the system, but this time, while traveling between a pair of positions $x$ and $y$ in $X$, we allow the pieces of the system to \textit{intertwine} with each other and not necessarily have them travel independently as in the case of $\dTC$. More precisely, during the motion, the weighted pieces are allowed to join with each other and break further into differently weighted pieces in a controlled way such that the unordered probability distribution of the motion from $x$ to $y$ remains unchanged. When two pieces join, their weights add up. When a piece splits into two or more pieces, its weight splits accordingly to give weights to the new pieces. Unlike the case of $\dTC$ where a specific motion of the system from $x$ to $y$ was obtained from a unique unordered collection of weighted pieces, this new approach allows multiple unordered collections of weighted pieces of the system to produce the same motion (see Section~\ref{Paths2} for various such examples). Therefore, we now measure the complexity of the motion planning using only the image of the system's motion rather than using the unordered collections of its weighted pieces that create that motion.

Of course, this approach offers more freedom of motion; therefore, one can expect to have even lesser navigational complexity using this approach. To measure it, the notion of the \textit{intertwining distributional topological complexity} of $X$, denoted $\iTC(X)$, was proposed in~\cite{DJ} as one less than the minimum value of $n$ for which the system can achieve a continuous motion by traveling as smaller pieces that can possibly intertwine with each other.
With this approach, a sequential motion of the system can also be planned in the same way as before, whose complexity can then be measured by the generalized notion of the $m^{th}$ \textit{sequential intertwining topological complexity} of $X$ that we will denote by $\iTC_m(X)$ and study in this paper. It is intuitively clear that $\iTC_m(X) \le \dTC_m(X)$.

As far as the physical applications of our intertwining invariants are concerned, we feel that the notion of $\iTC_m$ can potentially be useful in minimizing the navigational complexity of the (sequential) motion planning problem for large flocks of identical drones\hspace{0.5mm}\footnote{\hspace{0.5mm}A nice example of such a situation appears in Episode 6 ``Hated in the Nation" in Season 3 of the Netflix TV series ``Black Mirror". We thank A.~Dranishnikov for this reference.}, where an intertwining motion is more likely to occur.

\subsection{Comparison}\label{why}
At this stage, one may ask why do we need another notion of complexity and a new sequence of invariants? As mentioned above, our target is to improve robot motion planning by minimizing the complexity of motion planning for systems with a given configuration space. These new notions of $\dTC_m$ and $\iTC_m$ make it possible by offering lower bounds to the respective notion of $\TC_m$ and reducing complexity in some cases. For the simple problem of planning the motion of a rotating line in $\R^{k+1}$ that is fixed along a revolving joint at a base point, $\dTC$ offers a significantly better solution than $\TC$. While in the classical setting of $\TC$, the minimum number of rules required to plan such a motion equals one more than the immersion dimension of $\R P^k$ when $k \ne 1,3,7$ and $k+1$ otherwise~\cite{FTY}, in the distributional setting, only two rules are enough to plan this motion for all $k \ge 1$,~\cite{DJ},~\cite{KW}. 

Of course, another important reason is that these new notions bring with them their LS-category versions ($\icat$ and $\dcat$) and encourage us to get examples of spaces on which they disagree with the old notions, hence creating different theories. Several ways in which the theory of the distributional invariants differs from that of the classical invariants are highlighted in~\cite{DJ},~\cite{Ja},~\cite{KW}. In the same spirit, it was shown in~\cite[Section 6]{Dr} that $\iTC(\H) = 1$ for Higman's group $\H$ (which was introduced in \cite{Hi}). As a consequence, $\icat(\H) = 1$ is obtained, thereby showing that an Eilenberg--Ganea-type theorem~\cite{EG} does not hold in the case of intertwining invariants. It is noteworthy that a version of the Eilenberg--Ganea theorem was proven for torsion-free groups in the distributional case,~\cite{KW},~\cite{Dr}. To see more differences between the theories of the intertwining and the distributional invariants, we show in Section~\ref{higmansecction} of this paper that $\dTC_m(\H)$ is at least $2(m-1)$ whereas $\iTC_m(\H)=1$. So, there are cases where these new notions offer strict improvements over the previous notions as far as motion planning is concerned. Hence, it seems worthwhile to develop their properties and understand them better.

\subsection{Continuous motion planning algorithm}\label{delicate}
For a metric space $Z$, let $\mathcal{B}(Z)$ denote the set of probability measures on $Z$ and for any $n \ge 1$, let
\[
\mathcal{B}_{n}(Z) = \left\{\mu \in \mathcal{B}(Z) \mid |\supp(\mu)| \le n \right\}
\]
denote the space of probability measures on $Z$ supported by at most $n$ points, equipped with the Lévy--Prokhorov metric~\cite{Pr}. The support of any $\mu \in \mathcal{B}_n(Z)$ is given by $\supp(\mu) =\{z\in Z\mid \lambda_z>0\}$. There is an inclusion $\omega_n: Z \hookrightarrow\B_n(Z)$ defined as $\omega_n:x \to \delta_x$ that sends $x \in Z$ to the Dirac measure $\delta_x$ supported at $x$. 

When $Z$ is non-metrizable, an alternative description of the space $\B_n(Z)$ with another topology is given as follows. Let $*^n Z$ denote the $n^{th}$ iterated join of $Z$ in the sense of~\cite{KK}. The group $S_n$ acts on $*^n Z$ by permutation of coordinates. The orbit space of this action is the $n^{th}$ iterated symmetric join $\text{Sym}(*^n Z)$. Then as in~\cite{KK}, $\B_n(Z)=\text{Sym}(*^n Z)/\sim$ where $t_1x_1+t_2x_1+\cdots+t_nx_n \sim (t_1+t_2)x_1+\cdots+t_nx_n$. This is also called the barycenter space of $Z$. Let $q:\text{Sym}(*^n Z) \to \B_n(Z)$ be the quotient map of the above relation that sends unordered formal linear combinations of size $n$ to probability measures of support at most $n$. Using $q$, the quotient topology $\mathscr{T}_1$ is induced on $\B_n(Z)$. In~\cite{KW}, the topology $\mathscr{T}_1$ was considered while defining the analog invariants, which could be different from the metric topology $\mathscr{T}_2$ considered in~\cite{DJ},~\cite{Ja}. It was explained in~\cite[Section 8]{Ja} that the identity map $\mathcal{I}:(\B_n(Z),\mathscr{T}_1) \to (\B_n(Z),\mathscr{T}_2)$ is continuous when $Z$ is a metric space. Furthermore, if $Z$ is a locally finite CW complex or a discrete space, then it is possible to show that $\mathcal{I}$ is a homotopy equivalence,~\cite[Section 1]{Dr}. So, the definitions of the distributional invariants and the analog invariants agree on such spaces. Therefore, we can study the intertwining invariants on these spaces as lower bounds to their respective distributional invariants by working in the space $(\B_n(Z),\mathscr{T}_2)$.

Let $Z = P(X) = \{f \hspace{1mm} | \hspace{1mm} f: [0,1] \to X\}$ be the path space of $X$ with the compact-open topology, and $P(\ov{x}) = \{f \in P(X) \hspace{1mm} | \hspace{1mm} f(t_i) = x_i \text{ for all } 1 \le i \le m\}$ for any $\ov{x}  = (x_1,\ldots,x_m) \in X^m$, where $t_i = (i-1)/(m-1)$ for all $1 \le i \le m$. Also, for Dirac measures $\delta_{x_i}$, let
\[
P(\delta_{x_1},\ldots,\delta_{x_m}) = \{ g \in P(\B_n(X)) \mid g(t_i) = \delta_{x_i} \text{ for all } i\}.
\]
Consider the continuous map $\Phi_n:\B_n(P(X))\to P(\B_n(X))$
defined as
\[
\Phi_n\left(\sum \lambda_\alpha \alpha \right)(t) = \sum \lambda_\alpha \alpha(t)
\]
for all $t \in I=[0,1]$. Since we are concerned with the image of the system's motion as a consequence of repeated breaking and reassembling and possible intertwining of pieces, the desired sequential motion planning algorithm is a continuous assignment of each $\ov{x} \in X^m$ to a path in $P(\delta_{x_1},\ldots,\delta_{x_m})$. But since we also want the distribution of the sequential motion from $x_1$ to $x_m$ to remain unchanged, we additionally require our path in $P(\delta_{x_1},\ldots,\delta_{x_m})$ to have a preimage in $\B_n(P(X))$ under the map $\Phi_n$. Thus, the desired $n$-intertwined $m$-navigation algorithm on $X$ is a continuous map 
\[
s_m:X^m \to P(\B_n(X))
\]
such that for each $\ov{x} \in X^m$, we have $s_m(\ov{x}) \in P(\delta_{x_1},\ldots,\delta_{x_m})$ and $\Phi_n^{-1}(s_m(\ov{x})) \ne \emptyset$. Such paths in $P(\B_n(X))$ that have a preimage in $\B_n(P(X))$ under $\Phi_n$ are called \textit{resolvable paths} and their preimages are called their \textit{resolver measures}. Like the case of $\dTC_m$~\cite[Section 1.B]{Ja}, as we aim to minimize the maximum number of pieces $n$ into which the system can break, we obtain the notion of $\iTC_m$.

\subsection{About this paper} 
This paper is organized as follows. In Section~\ref{Preliminaries}, we revisit the theory of the classical and the distributional notions of LS-category and (sequential) topological complexity. In Section~\ref{Resolvable Paths}, we study resolvers and the space of resolvable paths. The intertwining invariants are formally defined and studied in Section~\ref{Intertwining Invariants} where we prove several properties for them, such as their homotopy invariance, behavior on products of spaces, and their relationship with each other and with their respective distributional counterparts. In Section~\ref{higmansecction}, we establish that the notions of $\dTC_m$ and $\iTC_m$ are different for each $m \ge 2$ by showing that $\iTC_m(\H) = 1$ for Higman's group $\H$. Section~\ref{Characterizations} contains some simple characterizations of $\iTC_m$ and $\icat$ in terms of pullback. In Section~\ref{Bounds}, we obtain lower bounds to $\icat(X)$ and $\iTC_m(X)$
using the cohomology of their respective symmetric squares. We also discuss the difficulty in obtaining better lower bounds. Finally, we end this paper in Section~\ref{Computations} by performing some computations using our lower bounds and providing examples of spaces $X$ for which $\icat(X) > 1$, $\iTC_m(X) >1$, $\icat(X) = \dcat(X)$, and $\iTC(X) = \dTC(X)$.

We use the following notations and conventions in this paper. All the topological spaces considered are path-connected ANR spaces. The composition of functions $f:X \to Y$ and $g:Y \to Z$ is denoted by $gf$. The symbol ``$=$" is used to denote homeomorphisms and isomorphisms, and the symbol ``$\simeq$" is used to denote homotopy equivalences of spaces and maps. For any fixed $m \ge 2$, we agree to set timestamps $t_i = (i-1)/(m-1)$ for all $1 \le i \le m$.

\section{Classical and distributional invariants}\label{Preliminaries}
We recall the classical definitions of the Lusternik--Schnirelmann category~\cite{CLOT} and the sequential topological complexity~\cite{Far1},~\cite{Rud},~\cite{BGRT} of a space.

The \emph{Lusternik--Schnirelmann category} (LS-category), $\cat (X)$, of $X$ is the minimal integer $n$ such that there is a covering $\{U_i\}$ of $X$ by $n+1$ open sets each of which is contractible in $X$.
    
For given $m \ge 2$, the $m^{th}$ \emph{sequential topological complexity}, $\TC_m(X)$, of $X$ is the minimal integer $n$ such that there is a covering $\{V_i\}$ of $X^m$ by $n+1$ open sets over each of which there is a continuous map $s_i : V_i \to P(X)$ such that for each $\ov{x} \in V_i \subset X^m$, we have that $s_i(\ov{x})(t_j) = x_j$ for all $1 \le j \le m$.

We also recall the definitions of the distributional LS-category~\cite{DJ} and the sequential distributional topological complexity~\cite{DJ},~\cite{Ja} of a space. 

\begin{definition}
    The {\em distributional LS-category}, $\dcat(X)$, of a space $X$, is the minimal integer $n$ for which there exists a map $H: X \to \B_{n+1}(P_0(X))$ such that $H(x) \in \B_{n+1}(P(x,x_0))$ for all $x \in X$.
\end{definition}

Here, the pointed space $(X,x_0)$ is considered, and $P_0(X)$ denotes the based path space containing paths $\gamma$ in $X$ with $\gamma(1)=x_0$.

\begin{definition}
    For a given $m \ge 2$, the {\em $m^{th}$ sequential distributional topological complexity}, $\dTC_m(X)$, of a space $X$, is the minimal integer $n$ for which there exists a map $s_m: X^m \to \B_{n+1}(P(X))$ such that $s_m(\ov{x}) \in \B_{n+1}(P(\ov{x}))$ for all $\ov{x} \in X^m$.
\end{definition}

The distributional invariants $\dcat$ and $\dTC_m$ satisfy~\cite{DJ},~\cite{Ja} most of the nice properties and relations as the classical invariants $\cat$ and $\TC_m$,~\cite{Far1},~\cite{Rud},~\cite{BGRT}. However, in general, $\dcat$ and $\cat$ are different notions, and similarly, so are $\dTC_m$ and $\TC_m$ for each $m \ge 2$. In particular, we have $\dcat(\R P^n) = 1 = \dTC(\R P^n)$ for all $n \ge 1$ due to~\cite[Section 3]{DJ}, which is in stark contrast to the equality $\cat(\R P^n)=n$ and the main result of~\cite{FTY}. Furthermore, $\dTC_m(\R P^n) \le 2m+1$ follows from~\cite[Section 8.A]{Ja} (see also~\cite[Corollary 6.6]{KW}) for each $m\ge 2$. As explained in Section~\ref{why}, these examples lead to various differences between the theories of the classical and the distributional invariants (see also Remark~\ref{obv}). Another such instance of that is mentioned in Remark~\ref{newnewnew}.

\subsection{Lower bounds for $\cat$ and $\TC_m$} 
Sharp lower bounds to these invariants come from cohomology. We first recall some basic notions. 

The \textit{cup length} of a space $X$ with coefficients in a ring $R$ (or alternatively, the cup length of the cohomology ring $H^*(X;R)$), denoted $c\ell_R(X)$, is the maximal length $k$ of a non-zero cup product $\alpha_1 \smile\cdots\smile\alpha_k\neq 0$ of cohomology classes $\alpha_i$ of positive dimensions. 

For a fixed $m \ge 2$, let $\Delta : X \to X^m$ be the diagonal map that induces the homomorphism $\Delta^*: H^*(X^m;R) \to H^*(X;R)$. The elements of the ideal $\Ker(\Delta^*)$ are called $m^{th}$ \emph{zero-divisors}, and the cup length of $\Ker(\Delta^*)$ is called the $m^{th}$ {\em zero-divisors cup length} of $X$, denoted $zc\ell^m_R(X)$. When $m = 2$, we denote it by $zc\ell_R(X)$.

It is well-known that for any ring $R$, $c\ell_R(X) \le \cat(X)$, see~\cite{CLOT}, and that $zc\ell^m_R(X) \le \TC_m(X)$, see~\cite{Far1},~\cite{Rud}, and~\cite{BGRT}.

\subsection{Lower bounds for $\dcat$ and $\dTC_m$}\label{pre2} For a space $Y$ and $k \ge 1$, its $k^{th}$ symmetric power $SP^{k}(Y)$ is defined as the orbit space of the action of the symmetric group $S_{k}$ on the product space $Y^k$ by permutation of coordinates. In this section, we regard each $[y_1, \ldots, y_k] \in SP^k(Y)$ as a formal sum $\sum n_i y_i$, where $n_i \ge 1$ and $\sum n_i = k$, subject to the equivalence $n_1 y + n_2 y = (n_1 + n_2)y$. We define the diagonal map $\pa^Y_k : Y \to SP^k(Y)$ as $\pa^Y_k (y) = [y, \ldots, y] = ky$. The following result in singular cohomology will be very useful later.

\begin{proposition}[\protect{\cite[Proposition 4.3]{DJ}}]\label{useful}
For a finite simplicial complex $Y$ and any $k \ge 1$, the induced homomorphism $(\pa_k^Y)^* : H^*(SP^k(Y);\F) \to H^*(Y;\F)$ is surjective if $\F \in \{\Q,\R\}$.
\end{proposition}

We note that Proposition~\ref{useful} was proven in~\cite{DJ} for rational coefficients but the same proof holds for real coefficients as well!

In this paper, we shall regard $Y$ as the subspace of $SP^k(Y)$ under the diagonal map $\pa_k^Y$ and use the term {\em inclusion} to refer to $\pa_k^Y$.

\begin{theorem}[\protect{\cite{DJ}}]
 Suppose that $\alpha_{i}^{*} \in H^{k_{i}}\left(SP^{n!}(X);R\right)$, $1 \leq i \leq n$, for some ring $R$ and $k_{i} \geq 1$. Let  $\alpha_{i} \in H^{k_{i}}(X;R)$ be the image of $\alpha_{i}^{*}$ under the induced homomorphism $(\pa^X_{n!})^*$ such that $\alpha_{1} \smile \cdots \smile \alpha_{n} \neq 0$. Then $\dcat(X) \geq n.$
 \end{theorem}

Let the map $\Delta_n : SP^{n!}(X) \to SP^{n!}(X^m)$ be induced from $\Delta$ by functoriality. Then the following lower bound is obtained in Alexander--Spanier cohomology.

\begin{theorem}[\protect{\cite{Ja}}]
Suppose that $\beta_i^* \in H^{k_i}(SP^{n!}(X^m); R)$, $1 \le i \le n$, for some ring $R$ and $k_i \ge 1$, are cohomology classes such that $\Delta_n^*(\beta_i^*) = 0$. Let $\beta_i$ be their images under the induced homomorphism $(\pa^{X^m}_{n!})^*$ such that $\beta_1 \smile \cdots \smile \beta_n \neq 0$. Then $\dTC_m(X) \ge n.$
\end{theorem}

From the above theorems and Proposition~\ref{useful}, we conclude that for any space $X$, we have $c\ell_{\F}(X)\le \dcat(X)$ and $zc\ell_{\F}(X)\le \dTC(X)$ for $\F\in\{\Q,\R\}$. We note that for a general ring $R$, $c\ell_{R}(X)$ and $zc\ell_{R}(X)$ are not necessarily the lower bounds to $\dcat(X)$ and $\dTC(X)$, respectively, since $\dcat(\R P^n) = \dTC(\R P^n)=1$.

\subsection{A characterization of $\dcat$ and $\dTC_m$} For a fibration $p:E \to B$, let $E_n(p) = \{ \mu \in B_n(E) \mid \supp(\mu) \subset p^{-1}(x) \text{ for some } x \in B \}$ and $\mathcal{B}_n(p) : E_n(p) \to B$ be the fibration defined as $\mathcal{B}_n(p)(\mu) = x$ when $\mu \in \mathcal{B}_n (p^{-1}(x)).$

Let $p_0:P_0(X) \to X$ and $\pi_m:P(X) \to X^m$ be the fibrations defined as 
$$
p_0(\phi) = \phi(0) \hspace{5mm} \text{and} \hspace{5mm} \pi_m(\psi)= \left(\psi\left(t_1\right), \psi\left(t_2\right), \ldots, \psi\left(t_m\right)\right).
$$
The following characterizations were obtained in~\cite{DJ} and~\cite{Ja}, respectively.
\begin{itemize}
    \item $\dcat(X) < n$ if and only if there is a section to  $\B_n(p_0):P_0(X)_n(p_0) \to X$.
    \item Similarly, $\dTC_m(X) < n$ if and only if there exists a section to the fibration $\B_n(\pi_m):P(X)_n(\pi_m) \to X^m$.
\end{itemize}
In view of these characterizations, $\dcat$ and $\dTC_m$ can be seen as special cases of the general notion of the ``distributional sectional category" of Hurewicz fibrations, see~\cite[Section 5]{Ja} for details.

\subsection{Some preliminaries}\label{pred}
For any $X$, $m \ge 2$, and $a_i \in (1,\infty)$ such that $a_i > a_{i+1}$ for all $1 \le i \le m-2$, let
$$
T_{m}(X) = \left\{\left(f_{1}, \ldots, f_{m} \right) \in (P(X))^{m} \mid f_{i}(1) = f_{i+1}(0) \text{ for all } 1 \le i \le m-1 \right\}
$$
and $\theta_{m}: T_{m}(X) \to P(X)$ be defined as $\theta_{m}\left(f_{1}, \ldots, f_{m} \right) = f_{1} \star \cdots \star f_{m}$, where
$$
\left(f_{1} \star \cdots \star f_{m}\right)(t) = \begin{cases}
    f_{1}(a_{1}t) & : 0 \le t \le \frac{1}{a_{1}} \\
    f_{2}\left(\frac{a_{2}(a_{1}t-1)}{a_{1}-a_{2}}\right) & : \frac{1}{a_{1}} \le t \le \frac{1}{a_{2}}
    \\
    \hspace{6mm} \vdots & \hspace{8mm} \vdots 
    \\
    f_{m-1}\left(\frac{a_{m-1}(a_{m-2}\hspace{0.5mm} t-1)}{a_{m-2}-a_{m-1}}\right) & : \frac{1}{a_{m-2}} \le t \le \frac{1}{a_{m-1}} 
    \\
    f_{m}\left(\frac{1-a_{m-1}\hspace{0.5mm} t}{1-a_{m-1}}\right) & : \frac{1}{a_{m-1}} \le t \le 1
\end{cases}
$$
The proof of the following statement can be found in~\cite[Lemma 2.1]{Ja}.
\begin{lemma}\label{iscont}
    The map $\theta_{m}:T_{m}(X) \to P(X)$ is continuous for each $m \ge 2$.
\end{lemma}

This construction was used in~\cite[Sections 3 and 5.B]{Ja} to prove various properties for $\dTC_m$. Here, we will use it in Sections~\ref{Intertwining Invariants} and~\ref{Characterizations} for similar purposes.

\section{Resolvable paths}\label{Resolvable Paths}
\subsection{Topological aspects}
For a fixed $n \ge 1$, recall the continuous mapping $\Phi_n: \B_n(P(X)) \to P(\B_n(X))$ defined in the introduction. The notion of resolvable paths was first introduced in~\cite{DJ} as follows.

\begin{definition}\label{resolve}
    A path $f: I \to \B_{n}(X)$ is called \textit{resolvable} if there exists a measure $\mu \in \B_{n}(P(X))$ such that $\Phi_n(\mu) =f$. In this case, $\mu$ is called \textit{a resolver of} $f$.
\end{definition}

Let $RP(\B_n(X)) \subset P(\B_n(X))$ denote the subspace of resolvable paths. 

\begin{remark}\label{basic}
Let $f \in RP(\B_n(X))$ be resolved by $\mu = \sum \lambda_\phi \phi \in \B_n(P(X))$. Let there exist some $t_0 \in I$ and $x \in X$ such that $f(t_0) = \delta_x$. Then
    $$
    \sum \lambda_\phi \phi(t_0)= f(t_0)=\delta_x
    $$
    implies that $\phi(t_0) = x$ for each $\phi \in \supp(\mu)$. 
\end{remark}

In~\cite[Section 6]{Dr}, the definition of resolvable paths was reformulated. The following lemma shows the equivalence of the two definitions.



\begin{lemma}\label{sames}
    If $f\in P(\B_n(X))$, then $f\in RP(\B_n(X))$ if and only if there exists a continuous map $F: I \times \{1,\ldots, n\} \to X$ and a measure $\nu \in \B_n(\{1,\ldots, n\})$ such that for all $t \in I$, $f(t) = \B_n(F)(t,\nu)$, where $\B_n(F): I \times \B_n(\{1,\ldots, n\}) \to \B_n(X)$ is defined using the functoriality of $\B_n$ in the obvious way.  
\end{lemma}

\begin{proof}
    Let us fix some $x_0 \in X$. For brevity, let $J'$ denote the unordered set $\{1,\ldots, n\}$. Let $f \in RP(\B_n(X))$ so that there exists $\mu = \sum_{i \in J}\lambda_{i}\phi_i \in \B_n(P(X))$ for some $J \subset J'$ such that $f(t) = \sum_{i \in J}\lambda_{i}\phi_i(t)$ for each $t \in I$. Let us define the map $F: I \times J \to X$ as $F(t,i) = \phi_i(t)$. When $i \in J'\setminus J$, let $F(t,i) = x_0$ for all $t$. This gives a continuous map $F:I \times J' \to X$. Also, define $\nu = \sum_{i\in J}\lambda_i \hspace{0.4mm} i \in \B_n(J) \subset \B_n(J')$. We see that
    $$
    \B_n(F) \left(t,\nu\right) = \sum_{i\in J}\lambda_i \hspace{0.4mm} F(t,i) = \sum_{i\in J}\lambda_i \hspace{0.4mm} \phi_i(t) = f(t).
    $$    
    Conversely, if there exists a map $F : I \times J' \to X$ and $\nu = \sum_{i\in J}\lambda_i \alpha_i \in \B_n(J')$ for some $J \subset J'$ such that $f(t) =  \sum_{i\in J}\lambda_i F(t,\alpha_i)$, then for each $i \in J$, define the path $\phi_i : I \to X$ as $\phi_i(t) = F(t,\alpha_i)$. Finally, we define $\mu = \sum_{i\in J}\lambda_i\phi_i \in \B_n(P(X))$. Clearly, $\Phi_n(\mu) = f$. Thus, $f \in RP(\B_n(X))$.
\end{proof}

We now study the homotopy type of the space $RP(\B_n(X))$ of resolvable paths.

\begin{lemma}\label{deform}
    For any $n \ge 1$, the space $RP(\B_n(X))$ deforms to $\B_n(X)$.
\end{lemma}

\begin{proof}
    Let us define a map $f:\B_n(X) \to P(\B_n(X))$ as $f(\mu) = \phi_\mu$, where $\phi_\mu$ is the trivial loop at $\mu \in \B_n(X)$, i.e., if $\mu = \sum \lambda_\alpha \alpha$, then $\phi_\mu(t) =  \sum \lambda_\alpha \alpha$. For any $x \in X$, let $c_x \in P(X)$ denote the trivial loop at $x$. If $\nu =  \sum \lambda_\alpha c_\alpha \in \B_n(P(X))$, then $\nu$ is a resolver of $\phi_\mu$. So, the image of $f$ is in $RP(\B_n(X))$. Define $g:RP(\B_n(X)) \to \B_n(X)$ to be the evaluation map $g(\psi) = \psi(0)$. Then clearly, $gf = \mathbbm{1}_{\B_n(X)}$. Finally, define a homotopy $H: RP(\B_n(X)) \times I \to P(\B_n(X))$ such that
    $$
    H(\psi,t)(s) = \psi(s(1-t))
    $$
    for all $s,t \in I$. Of course, we get $H(\psi,0)(s) = \psi(s)$ and $H(\psi,1)(s) = \psi(0)$ for all $s \in I$. Now, see that if $\mu = \sum \lambda_\alpha \alpha \in \B_n(P(X))$ is a resolver of $\psi \in RP(\B_n(X))$ and for each $t \in I$, if $\gamma_t:I \to I$ is given by $\gamma_t(s) = s(1-t)$, then the measure $\nu = \sum \lambda_\alpha \hspace{0.3mm}\alpha\gamma_t \in \B_n(P(X))$ is a resolver of $H(\psi,t) \in P(\B_n(X))$. Hence, the image of $H$ is in $RP(\B_n(X))$ and so, $H$ is a homotopy between $\mathbbm{1}_{RP(\B_n(X))}$ and $fg$. 
\end{proof}

\begin{corollary}\label{useonce}
    The space $RP(\B_n(X))$ is homotopy equivalent to $P(\B_n(X))$.
\end{corollary}

\begin{proof}
    We know that for any space $Y$, the path space $P(Y)$ deforms to $Y$. Upon taking $Y = \B_n(X)$ and using Lemma~\ref{deform}, the statement follows.
\end{proof}

For a metric space $(X,d)$ and $n \ge 1$, we define $\exp_n(X)=\{A \subset X \mid |A|\le n\}$ to be the set of all subsets of $X$ having cardinality at most $n$. We equip $\exp_n(X)$ with the topology induced by the Hausdorff metric $d_H$. So, for any $A,B \in \exp_n(X)$, 
$$
d_H(A,B) = \max\left\{ \sup_{a \in A} d(a,B), \hspace{0.5mm }\sup_{b \in B} d(A,b) \right\}.
$$
It is easy to see that $\exp_1(X) = X$,  $\exp_2(X)=SP^2(X)$, and $\exp_n(X) \subset \exp_{n+1}(X)$ for all $n \ge 1$.

Consider a support function $\supp:\B_n(X) \to \exp_n(X)$ that maps each measure $\mu \in \B_n(X)$ to its support $\supp(\mu) \in \exp_n(X)$. For $n \ge 2$, this function need not be continuous. For any $f \in P(\B_n(X))$, the composition $\supp \hspace{0.3mm} f : I \to \exp_n(X)$ can be viewed as the variation of the support of $f$ with respect to the time $t$.

\begin{proposition}
    If $f \in RP(\B_n(X))$, then the composition $\supp f$ is continuous.
\end{proposition}

\begin{proof}
    The continuity of $\supp f$ can be checked by the sequential criterion. Let $\{t_n\}$ be a sequence in $I$ converging to $t$. Let $A = (\supp f)(t) = \supp(f(t))$ and for each $n$, let $A _n = (\supp f)(t_n) = \supp(f(t_n))$. Since $f \in RP(\B_n(X))$, there exists $\mu = \sum \lambda_\phi \phi \in \B_n(P(X))$ such that $f(s) = \sum \lambda_\phi \phi(s)$ for all $s \in I$. Let us fix some $\phi \in \supp(\mu)$. Then $\phi(t_n) \in A_n$ and $\phi(t) \in A$. Thus, for all $n$, 
    $$
    d(A_n,\phi(t)), \hspace{1mm}d(\phi(t_n),A) \le d(\phi(t_n),\phi(t)).
    $$
    Now, choose some $\epsilon > 0$. Since $\phi$ is continuous, there exists some $m_\phi \ge 1$ such that $d(\phi(t_n),\phi(t)) < \frac{\epsilon}{2}$ for all $n\ge m_\phi$. This happens for each fixed $\phi \in \supp(\mu)$. Let $m = \max\{m_\phi \mid \phi \in \supp(\mu)\}$. Then, it follows that for all $n \ge m$,
    $$
    \sup_{\phi \in \supp(\mu)} d(A_n,\phi(t)),\sup_{\phi \in \supp(\mu)} d(\phi(t_n),A) \le \frac{\epsilon}{2}.
    $$
    Note that $A$ is a subset of the multiset $\{\phi(t)\mid\phi \in \supp(\mu)\}$ and $A_n$ is a subset of the multiset $\{\phi(t_n)\mid\phi \in \supp(\mu)\}$ for each $n$. Thus, for each $n$,
    $$
    \sup_{a \in A} d(A_n,a) \le \sup_{\phi \in \supp(\mu)} d(A_n,\phi(t)) \hspace{4mm} \text{and} \hspace{4mm} \sup_{b \in A_n} d(b,A) \le \sup_{\phi \in \supp(\mu)} d(\phi(t_n),A).
    $$
    From this, we conclude that $d_H(A_n,A) \le \frac{\epsilon}{2} < \epsilon$ for all $n \ge m$. So, the sequence $\{A_n\}$ converges to $A$ in $\exp_n(X)$ and hence, the map $\supp f$ is continuous. 
\end{proof}

In other words, the support of a resolvable path varies continuously with time.

\subsection{Some examples}\label{Paths2}
Let us look at various explicit examples of the resolvers of some resolvable paths in $RP(\B_n(X))$ for $n=2,3,$ and $4$. Through these examples, we will not only see that $\Phi_n$ is not injective but also understand that resolvers of a given resolvable path may have wildly different characteristics.

For any two paths $\phi_1,\phi_2\in P(X)$, if there exists some $t_1\in(0,1)$ such that $\phi_1(t_1) = \phi_2(t_1)$, then we define a path $(\phi_1,\phi_2) \in P(X)$ by pasting such that $(\phi_1,\phi_2)(t) = \phi_1(t)$ when $t \in [0,t_1]$ and $(\phi_1,\phi_2)(t) = \phi_2(t)$ when $t\in [t_1,1]$. We note that the definition of $(\phi_1,\phi_2)$ depends on the choice of $t_1$.

\begin{example}\label{1}
    Consider a resolvable path $f:I \to \B_2(X)$ described as follows.
At $t =0$, the path breaks into two strings of weights $\frac{1}{2}$ and $\frac{1}{2}$. At $t=t_1$, the strings intertwine and break into two strings of weights $\frac{1}{2}$ and $\frac{1}{2}$. Then the strings reassemble at $t=1$. Let $\mu = \frac{1}{2}\alpha + \frac{1}{2}\beta$ be a resolver with $\alpha(t_1) = \beta(t_1)$ such that $f(t) = \frac{1}{2}\alpha(t) + \frac{1}{2}\beta(t)$. Then, $\nu = \frac{1}{2}(\alpha,\beta) +\frac{1}{2}(\beta,\alpha)$ is another resolver of $f$.
\end{example}

\begin{example}\label{2}
    Consider a path $g$, similar to the one described in Example~\ref{1}, the weights of whose two strings change from $\frac{1}{2}$, $\frac{1}{2}$ to $\frac{1}{4}$, $\frac{3}{4}$. This path cannot have a resolver in $\B_2(P(X))$. Let $\mu = \frac{1}{4}\alpha + \frac{1}{4}(\alpha,\gamma) + \frac{1}{2}\gamma$ be a resolver of $g$, where $\alpha(t_1) = \gamma(t_1)$. Then clearly, $\nu = \frac{1}{2}(\alpha,\gamma)+\frac{1}{4}(\gamma,\alpha)+\frac{1}{4}\gamma$ is another resolver of $g$.
\end{example}

For three paths $\phi_1,\phi_2,\phi_3\in P(X)$ such that $\phi_1(t_1) = \phi_2(t_1)$ and $\phi_2(t_2) = \phi_3(t_2)$ where $0<t_1 < t_2<1$, we define a path $(\phi_1,\phi_2,\phi_3) \in P(X)$ by pasting in a similar way as above such that $(\phi_1,\phi_2,\phi_3)(t) = \phi_i(t)$ if $t \in [t_{i-1},t_i]$ for $i=1,2,3$, where we have $t_0=0$ and $t_3=1$. Again, $(\phi_1,\phi_2,\phi_3)$ depends on the choice of $t_1$ and $t_2$.

\begin{example}\label{3}
Let $h:I \to \B_4(X)$ be a resolvable path described as follows. At $t=0$, the path breaks into three strings of weights $\frac{1}{3}$ each. At $t=t_1$, the strings intertwine and break into three strings of weights $\frac{1}{2}$, $\frac{1}{3}$, and $\frac{1}{6}$. Then the strings reassemble at $t=1$. This path cannot have a resolver in $\B_3(P(X))$. Let $\mu = \frac{1}{3}\alpha + \frac{1}{6}(\gamma, \alpha) + \frac{1}{6}\gamma + \frac{1}{3}\delta \in \B_4(P(X))$ be a resolver of $h$, where $\gamma(t_1)=\alpha(t_1)$. Then, three more resolvers of $h$ in $\B_4(P(X))$ are: $\frac{1}{3}(\gamma,\alpha)+\frac{1}{6}\alpha+\frac{1}{6}(\alpha,\gamma) + \frac{1}{3}\delta$, $\frac{1}{3}(\delta,\alpha)+\frac{1}{6}\gamma+\frac{1}{6}(\gamma,\alpha) + \frac{1}{3}(\alpha,\delta)$, $\frac{1}{3}(\delta,\alpha)+\frac{1}{6}(\alpha,\gamma)+\frac{1}{6}\alpha+ \frac{1}{3}(\gamma,\delta)$.
\end{example}

Given a resolvable path $\psi$, let us call a point $t_0 \in (0,1)$ a \textit{branching point} of $\psi$ if there exists an $\epsilon>0$ such that two or more strings in $(t_0-\epsilon,t_0)$ intertwine at $t_0$ or a string breaks at $t_0$ into two or more strings in $(t_0,t_0+\epsilon)$.

\begin{example}\label{4}
    We modify the path $h$ above by adding another branching point. Let $\psi:I \to \B_4(X)$ be a resolvable path described as follows. At $t=0$, the path breaks into three strings of weights $\frac{1}{3}$ each. At $t=t_1>0$, the strings intertwine and break into three strings of weights $\frac{1}{2}$, $\frac{1}{3}$, and $\frac{1}{6}$. At $t=t_2>t_1$, the strings intertwine and break into two strings of weights $\frac{2}{3}$ and $\frac{1}{3}$. Then the strings reassemble at $t=1$. Again, this path cannot have a resolver in $\B_3(P(X))$. Let $\mu = \frac{1}{3}\alpha + \frac{1}{6}(\gamma,\alpha, \alpha) + \frac{1}{6}(\gamma,\gamma,\alpha) + \frac{1}{3}\delta$ be a resolver of $\psi$, where $\gamma(t_i)=\alpha(t_i)$ for $i=1,2$. It is easy to find at least $11$ other resolvers of $\psi$ in $\B_4(P(X))$.
\end{example}

From Examples~\ref{3} and~\ref{4}, we see that increasing the number of branching points of a resolvable path by one can rapidly increase the number of its resolvers. 

\begin{example}
    Consider the resolvable path $f:I \to B_4(X)$ described in Example~\ref{1}. Certainly, we have two of its resolvers $\mu,\nu\in\B_4(P(X))$ defined above. Another resolver $\delta \in \B_4(P(X))$ is defined as follows: $\delta = \frac{1}{4}\alpha + \frac{1}{4}\beta + \frac{1}{4}(\alpha,\beta)+ \frac{1}{4}(\beta,\alpha)$. We note that in $\B_3(P(X))$, however, $f$ has exactly two resolvers. 
\end{example}

So, even if we keep the number of branching points of a resolvable path fixed, its number of resolvers can increase as we increase the value of the maximum possible size $n$ of the support of its resolver measures. We also note that $f$ is a resolvable path whose resolvers have supports of different sizes. Indeed, $|\supp(\mu)|=2=|\supp(\nu)|$ and $|\supp(\delta)|=4$.


\section{The intertwining invariants}\label{Intertwining Invariants}
\subsection{Sequential intertwining complexities} For $\ov{x} = (x_1,\ldots, x_m) \in X^m$, recall that $P(\delta_{x_1},\ldots,\delta_{x_m}) = \{ f \in P(\B_n(X)) \mid f(t_i) = \delta_{x_i} \text{ for all } i\}$.
\begin{definition}
    An \textit{$n$-intertwined $m$-navigation algorithm} on a space $X$ is a map 
    $$
    s_m : X^m \to RP(\B_n(X))
    $$
    satisfying $s_m(\ov{x}) \in P(\delta_{x_1}, \ldots, \delta_{x_m})$
    for each $\ov{x}\in X^m$.
\end{definition}

In view of Remark~\ref{basic}, it is easy to see that the path $s_m(\ov{x}) \in RP(\B_n(X))$ is resolved by measures in $\B_n(P(\ov{x}))$.

\begin{definition}
    The $m^{th}$ \textit{sequential intertwining topological complexity}, $\iTC_m(X)$, of a space $X$ is the minimal integer $n$ such that $X$ admits an $(n+1)$-intertwined $m$-navigation algorithm.
\end{definition}

It is clear that $\iTC_m(X) = 0$ for some $m \ge 2$ if and only if $X$ is contractible.

When $m = 2$, $\iTC_2(X)$ is denoted by $\iTC$  and simply called the intertwining topological complexity of $X$,~\cite{DJ}.

We now explicitly prove some basic properties for $\iTC_m$ that are very similar to the properties of $\dTC_m$ (see~\cite[Section 3]{Ja}) and $\TC_m$ (see~\cite[Section 3]{Rud} and~\cite[Section 3]{BGRT}). The proofs are partly inspired by the corresponding proofs for $\dTC_m$ obtained in~\cite[Section 3]{Ja}. 

\begin{proposition}\label{domin}
    If $f:X \to Y$ is a homotopy domination, then for each $m \ge 2$, $\iTC_m(Y) \le \iTC_m(X)$.
\end{proposition}
\begin{proof}
    Since $f:X \to Y$ is a homotopy domination, there exists a continuous map $g:Y \to X$ such that $fg \simeq \mathbbm1_Y$. Let $\iTC_m(X) = n-1$. Then there exists a $n$-intertwined $m$-navigation algorithm $s_m:X^m \to RP(\B_n(X))$. Let $g^m:Y^m \to X^m$ be the product map and let $f_*: \B_n(X) \to \B_n(Y)$ and $\wt{f} : P(\B_n(X)) \to P(\B_n(Y))$ be induced by $f$ due to functoriality. Consider the map
    $$
    \wt{f} s_m \hspace{0.3mm}g^m :Y^m \to P(\B_n(Y)).
    $$
    Let $h:Y \to P(Y)$ be a homotopy such that $h(y) = h_y$ for each $y \in Y$, where $h_y(0) = y$ and $h_y(1) = fg(y)$. We write $H_y = \omega_n h_y:I \to \B_n(Y)$ as a path in $\B_n(Y)$ from $\delta_y$ to $\delta_{fg(y)}$, where we recall that $\omega_n:Y \hookrightarrow \B_n(Y)$ is the Dirac inclusion. For each $1 \le i \le m-1$, let $k_i: I \to I$ be the map
    $$
    k_i(s) = \frac{s+i-1}{m-1}.
    $$
    For a fixed $\ov{y} \in Y^m$, let $s_m \hspace{0.3mm}g^m(\ov{y}) = \psi$ and $\psi_i=\psi k_i$. In the notations of Section~\ref{pred}, using the map $\theta_{m-1}:T_{m-1}(\B_n(Y)) \to P(\B_n(Y))$ with $a_{i} = 1/t_{i+1} = (m-1)/i$ for all $1 \le i \le m-2$, we define a map $\sigma:Y^m \to P(\B_n(Y))$ as
    $$
\sigma(\ov{y}) = \left(H_{y_1} \cdot f_*\psi_1 \cdot \ov{H}_{y_2}\right) \star  \left(H_{y_2} \cdot f_*\psi_2 \cdot \ov{H}_{y_3}\right) \star \cdots \star \left(H_{y_{m-1}} \cdot f_*\psi_{m-1} \cdot \ov{H}_{y_m}\right).
$$
Here, $\cdot$ denotes the operation of the usual concatenation of paths. Continuity of $\sigma$ follows from the continuity of $\theta_{m-1}$ in Lemma~\ref{iscont}.
Clearly, $\sigma(\ov{y}) \in P(\delta_{y_1}, \ldots,\delta_{y_m})$. We now show that $\sigma(\ov{y})$ is resolvable. Let $\psi \in RP(\B_n(X))$ be resolved by $\mu = \sum \lambda_{\alpha} \alpha$. Thus, $\alpha(t_i) = g(y_i)$ for all $i$. As before, define
$$
\wt{\alpha} = \left(h_{y_1} \cdot f\alpha k_1 \cdot \ov{h}_{y_2}\right) \star  \left(h_{y_2} \cdot f\alpha k_2 \cdot \ov{h}_{y_3}\right) \star \cdots \star \left(h_{y_{m-1}} \cdot f\alpha k_{m-1} \cdot \ov{h}_{y_m}\right).
$$
Then the measure $\nu = \sum \lambda_{\alpha} \wt{\alpha} \in \B_n(P(Y))$ is a resolver of $\sigma(\ov{y})$. So, $\sigma$ is an $n$-intertwined $m$-navigation algorithm on $Y$. Thus, $\iTC_m(Y) \le n-1$.
\end{proof}
\begin{corollary}\label{hinv}
    $\iTC_m$ is a homotopy invariant of spaces for each $m \ge 2$.
\end{corollary}
\begin{proof}
    Let $X\simeq Y$. Then there exist continuous maps $f:X \to Y$ and $g:Y \to X$ such that $fg \simeq \mathbbm{1}_Y$ and $gf \simeq \mathbbm{1}_X$. Since $f$ and $g$ are homotopy dominations of each other, the conclusion follows from Proposition~\ref{domin}.
\end{proof}

\begin{corollary}\label{easyyy}
    $\max\{\iTC_m(X),\iTC_m(Y)\} \le \min\{\iTC_m(X \vee Y),\iTC_m(X \times Y)\}$ for all $m \ge 2$ and spaces $X$ and $Y$.
\end{corollary}

\begin{proof}
    Let us fix $z_0$ as the wedge basepoint of $X \vee Y$. Let $r_X:X \vee Y \to X$ and $r_Y:X \vee Y \to Y$ be the collapsing maps such that $r_X(Y) = \{z_0\}$ and $r_Y(X) = \{z_0\}$. Let $\iota_X:X \hookrightarrow X \vee Y$ and $\iota_Y:Y \hookrightarrow X \vee Y$ be the inclusions. Then, $r_X$ and $r_Y$ are homotopy dominations of $\iota_X$ and $\iota_Y$, respectively. Hence, we get from Proposition~\ref{domin} that $\iTC_m(X)\le \iTC_m(X \vee Y)$ and $\iTC_m(Y)\le \iTC_m(X \vee Y)$. Next, let $\proj_X:X \times Y \to X$ and $\proj_Y:X \times Y \to Y$ be the projection maps. Fix some $a \in X$ and $b \in Y$ and let $\iota'_X:X \hookrightarrow X \times Y$ and $\iota'_Y:Y \hookrightarrow X \times Y$ be the inclusions $\iota'_X(x)= (x,b)$ and $\iota'_Y(y)= (a,y)$. Then, $\proj_X$ and $\proj_Y$ are homotopy dominations of $\iota'_X$ and $\iota'_Y$, respectively. The inequalities $\iTC_m(X)\le \iTC_m(X \times Y)$ and $\iTC_m(Y)\le \iTC_m(X \times Y)$ then follow from Proposition~\ref{domin}.
\end{proof}

Corollary~\ref{easyyy} can also be proven independently, without using Proposition~\ref{domin}, by using ideas from~\cite[Proposition 3.8]{Ja}.

\begin{proposition}\label{ine1}
    For any $m \ge 2$, $\iTC_m(X) \le \dTC_m(X) \le \TC_m(X)$.
\end{proposition}
\begin{proof}
    Let $\dTC_m(X) = n -1$. Then there exists a map $k_m:X^m \to \B_n(P(X))$ such that $k_m(\ov{x}) \in \B_n(P(\ov{x}))$. Define $s_m = \Phi_n \hspace{0.2mm} k_m :X^m \to P(\B_n(X))$. By definition, $s_m(\ov{x}) \in RP(\B_n(X))$. Since $\Phi_n$ maps $\B_n(P(\ov{x}))$ to $P(\delta_{x_1}, \ldots, \delta_{x_m})$, the map $s_m$ is an $n$-intertwined $m$-navigation algorithm. Hence, $\iTC_{m}(X) \le n-1 = \dTC_m(X)$. The inequality $\dTC_m(X) \le \TC_m(X)$ is well-known, see~\cite{DJ},~\cite{Ja}.
\end{proof}

\begin{corollary}\label{vvobv}
    If $\dTC_m(X) = 1$ for any $m \ge 2$, then $\iTC_m(X) = 1$. 
\end{corollary}

\begin{example}
    $\iTC(\R P^n) = 1$ and $\iTC_m(\R P^n) \le 2m+1$  for all $n \ge 1$ and $m \ge 2$. 
\end{example}

\begin{remark}\label{obv}
In the classical theory, it is known due to~\cite[Corollary 3.5]{GLO1} that $\TC(X) = 1 \iff X \simeq S^{2k-1}$ for some $k$. As noted in~\cite[Section 3.3]{DJ}, this does not remain valid for $\dTC$ due to the example of $\R P^n$ for $n \ge 2$. For the same reason, this also does not hold in the case of $\iTC$. Furthermore, as noted in~\cite[Remark 3.12]{Ja} for $\dTC_m$, the conclusion of~\cite[Theorem 2.1]{FO} from the classical theory of $\TC_m$ does not hold in the theory of $\iTC_m$ as well because $\iTC_m(\R P^\infty) \le 2m+1$ whereas the cohomological dimension of any finite group is infinite.
\end{remark}

\subsection{Intertwining category} For a fixed basepoint $x _0 \in X$, let us define a subspace $RP_0(\B_k(X))  = \{f \in RP(\B_k(X)) \mid 
f(1) = \delta_{x_0} \}$. Then in the same spirit as $\iTC_m$, an analogous intertwining version of the LS-category can be defined.

\begin{definition}[\protect{\cite{DJ,Dr}}]
    The \textit{intertwining Lusternik--Schnirelmann category}, $\icat(X)$, of a space $X$ is the minimal integer $n$ such that there exists a continuous map $H : X \to RP_0(\B_{n+1}(X))$ satisfying $H(x)(0) = \delta_x$ for all $x \in X$.
\end{definition}

Of course, $H(x)(1) = \delta_{x_0}$ and $H(x)$ is resolved by measures in $\B_n(P(x,x_0))$.

Each of the above results for $\iTC_m$ holds for $\icat$ as well and the proofs are just minor modifications of the respective proofs for $\iTC_m$.

\begin{proposition}\label{simple}
    \begin{enumerate}
        \itemsep-0.25em 
\item $\icat(X) \le \dcat(X) \le \cat(X)$.
    \item If $f:X \to Y$ is a homotopy domination, then $\icat(Y) \le \icat(X)$.
        \item $\icat$ is a homotopy invariant of spaces. 
        \item $\max\{\icat(X),\icat(Y)\} \le \min\{\icat(X \vee Y),  \icat(X \times Y)\}$.
    \end{enumerate}
\end{proposition}

By definition, $\icat(X) = 0$ if and only if $X$ is contractible. Therefore, by Proposition~\ref{simple} (1), $\dcat(X) = 1 \implies \icat(X) = 1$. 

\begin{example}
    For any $n \ge 1$, $\icat(S^n) = 1 = \icat(\R P^n)$.
\end{example}

\begin{remark}\label{newnewnew}
    It follows from Proposition~\ref{simple} (1) and~\cite[Propositions 3.5 and 3.9]{Ja} that
    \[
    \icat(\R P^k \times \R P^k) \le \dcat(\R P^k \times \R P^k) \le \dTC_3(\R P^k) \le 3
    \]
    for any $k \ge 1$. For a path-connected CW complex $X$,~\cite[Proposition 2.2]{GLO1} says that if $\pi_1(X)$ has non-trivial torsion, then $\cat(X^n) \ge 2n$. Since $\pi_1(\R P^k) = \Z_2$, $\icat(\R P^k)$ and $\dcat(\R P^k)$ are \textit{not} at least $2$, and $\icat((\R P^k)^2)$ and $\dcat((\R P^k)^2)$ are \textit{not} at least $4$, we conclude that an analog of~\cite[Proposition 2.2]{GLO1} does not hold in the case of $\icat$ or $\dcat$. 
\end{remark}

\begin{lemma}[\protect{\cite[Proposition 3.3]{KW}}]\label{notmine}
    If $p:E \to X$ is a degree $k$ covering map, then the map $p^*:X \to \B_k(E)$ defined as
    $$
    p^*(x) = \frac{1}{k}\hspace{0.6mm} \sum_{\wt{x} \in p^{-1}(x)} \delta_{\wt{x}}
    $$
    is continuous and $p_* p^* = \omega_k : X \to \B_k(X)$, where $p_*: \B_k(E) \to \B_k(X)$.
\end{lemma}

The following result is inspired by~\cite[Proposition 6.5]{KW}.

\begin{proposition}\label{ine3}
    If $p:E \to X$ is a degree $k$ covering map, then for all $m \ge 2$, $\iTC_m(X) \le km(\iTC_m(E) + 1) - 1$ and $\icat(X) \le k(\icat(E) + 1) - 1$.
\end{proposition}

\begin{proof}
    First, we prove this statement for $\icat$. Let $\icat(E) = n-1$ so that there exists a map $H: E \to RP_0(\B_n(E))$. This induces $H_*:\B_k(E) \to \B_k(RP_0(\B_n(E)))$. Furthermore, we have our usual map $\Phi_k: \B_k(RP_0(\B_n(E))) \to P_0(\B_k(\B_n(E)))$. But $\B_k(\B_n(E))=\B_{nk}(E)$. So, $\Phi_k: \B_k(RP_0(\B_n(E))) \to P_0(\B_{nk}(E))$. Also, $p$ induces $\wt{p}:P(\B_{nk}(E)) \to P(\B_{nk}(X))$. Finally, define
    $$
    \sigma= \wt{p} \hspace{0.4mm}\Phi_k \hspace{0.4mm} H_* \hspace{0.3mm} p^*:X \to P(\B_{nk}(X)).
    $$
    This composition is continuous because $p^*$ is continuous due to Lemma~\ref{notmine}. By definition, $\sigma(x)(0) =\delta_x \in \B_{nk}(X)$ and $\sigma(x)(1) = \delta_{x_0} \in \B_{nk}(X)$. Since $H(x)$ is resolvable, $H_*$ is induced by $H$, and $\sigma$ factors through $\Phi_k$, it follows that $\sigma(x)$ is resolvable. Therefore, $\icat(X) \le kn-1$.

    To prove this for $\iTC_m$ for a fixed $m \ge 2$, we first observe that the product map $p_m:E^m \to X^m$ defined as $p_m(\ov{x}) = (p(e_1),\ldots,p(e_m))$ is a covering map of degree $km$. Thus, the induced map $p_m^*:X^m \to \B_{km}(E^m)$ is continuous by Lemma~\ref{notmine}. Now, letting $\iTC_m(E) = q-1$, there exists a $q$-intertwined $m$-navigation algorithm $s:E^m \to RP(\B_q(E))$ that induces by functoriality a map $s_*: \B_{km}(E^m) \to \B_{km}(RP(\B_q(E)))$. As before, we form the composition 
    $$
    \sigma_m = \wt{p} \hspace{0.4mm}\Phi_k \hspace{0.4mm} s_* \hspace{0.3mm} p_m^*:X^m \to P(\B_{qkm}(X)).
    $$
    It follows by construction that $\sigma_m(\ov{x}) \in P(\delta_{x_1},\ldots,\delta_{x_m})$ for each $\ov{x} \in X^m$ and the image of $\sigma_m$ is indeed in $RP(\B_{qk}(X))$. Therefore, $\iTC_m(X) \le kmq-1$. 
\end{proof}

\subsection{Relations between $\icat$ and $\iTC_m$} The following statements show that $\iTC_m(X)$ relates with the $\icat$ of products of $X$ in \textit{mostly} the same way as $\dTC_m(X)$ (resp. $\TC_m(X)$) relates with the $\dcat$ (resp. $\cat$) of products of $X$,~\cite{Ja},~\cite{Rud},~\cite{BGRT}.

\begin{proposition}\label{ine5}
    For any space $X$, $\icat(X)\le \iTC(X)$.
\end{proposition}

\begin{proof}
    Let $\iTC(X) = n-1$. Then there exists a $n$-intertwined navigation algorithm $s:X \times X \to RP(\B_n(X))$. For a fixed basepoint $x_0 \in X$, let $J:X \to X \times X$ be the inclusion $J(x)= (x,x_0)$. Taking the composition $sJ:X \to RP(\B_n(X))$, we see that $sJ(x)(0) = \delta_x$ and $sJ(x)(1) = \delta_{x_0}$ for all $x \in X$. Hence, $\icat(X) \le n -1$. 
\end{proof}

\begin{remark}
    We do not know if the inequality $\icat(X^{m-1}) \le \iTC_m(X)$ holds for $m \ge 3$. The proof of Proposition~\ref{ine5} for $m=2$ does not extend to $m \ge 3$ in analogy with the proof of~\cite[Proposition 3.5]{Ja}: given $\ov{x}=(x_1,\ldots,x_{m-1})\in X^{m-1}$ and basepoint $x_0 \in X$, we do not how to obtain, in a continuous way, a resolvable path in $P(\delta_{\ov{x}},\delta_{\ov{x_0}})\subset P(\B_n(X^{m-1}))$ from a given resolvable path in $P(\delta_{x_1},\ldots,\delta_{x_{m-1}},\delta_{x_0})\subset P(\B_n(X))$. One reason for this is that when $n \ge 2$ and $m \ge 3$, there is no non-trivial continuous map from $P((\B_n(X))^{m-1})$ to $P(\B_n(X^{m-1}))$ that preserves resolvability.
\end{remark}

\begin{proposition}\label{ine4}
    For any $m \ge 2$, $\iTC_m(X) \le \icat(X^m)$. 
\end{proposition}

\begin{proof}
    Let $\icat(X^m) =n-1$. There exists a map $H:X^m \to RP_0(\B_n(X^m))$ with $H(\ov{x})(0) = \delta_{\ov{x}}$ and $H(\ov{x})(1) = \delta_{\ov{x_0}}$ for all $\ov{x} \in X^m$ and a fixed basepoint $\ov{x_0} = (x_1^0,\ldots,x_m^0) \in X^m$. Let $\proj_i:X^m \to X$ be the projection onto the $i^{th}$ coordinate that induces the map $(\proj_i)_*:\B_n(X^m) \to \B_n(X)$. Let $H(\ov{x}) = \phi$, and define $\phi_i = (\proj_i)_*\phi$ and $\ov{\phi_i} = (\proj_i)_* \ov{\phi}$, where $\ov{\phi}$ is the reverse path of $\phi$. For each $i$, let $\gamma_i \in P(X)$ be a path such that $\gamma_i(0) = x_i^0$ and $\gamma_i(1) = x_{i+1}^0$. This gives $\gamma_i^* = \omega_n \gamma_i \in P(\B_n(X))$. Then using Lemma~\ref{iscont}, we define a continuous map $\sigma:X^m \to P(\B_n(X))$ as
    $$
    \sigma(\ov{x}) = \left(\phi_1 \cdot \gamma_1^* \cdot \ov{\phi}_2 \right) \star \left(\phi_2 \cdot \gamma_2^* \cdot \ov{\phi}_3 \right) \star \cdots \star \left(\phi_{m-1} \cdot \gamma_{m-1}^* \cdot \ov{\phi}_m \right).
    $$
Clearly, $\sigma(\ov{x}) \in P(\delta_{x_1}, \ldots, \delta_{x_m})$. We still need to verify the resolvability of $\sigma(\ov{x})$. Let $\mu = \sum \lambda_\alpha \alpha \in \B_n(P(X^m))$ be a resolver of $\phi \in RP_0(\B_n(X^m))$. Let $\alpha_i = \proj_i \alpha$ and $\ov{\alpha_i} = \proj_i \ov{\alpha}$, where $\ov{\alpha}$ is the reverse path of $\alpha$. As before, we define a joined path
$$
\wt{\alpha} = \left(\alpha_1 \cdot \gamma_1 \cdot \ov{\alpha}_2 \right) \star \left(\alpha_2 \cdot \gamma_2 \cdot \ov{\alpha}_3 \right) \star \cdots \star \left(\alpha_{m-1} \cdot \gamma_{m-1} \cdot \ov{\alpha}_m \right).
$$
Then the measure $\nu = \sum \lambda_\alpha \wt{\alpha} \in \B_n(P(X))$ is a resolver of $\sigma(\ov{x})$. Hence, $\sigma$ defines an $n$-intertwined $m$-navigation algorithm on $X$ and so, $\iTC_m(X) \le n-1$.
\end{proof}

\begin{proposition}\label{ineq6}
    For each $m \ge 2$, $\iTC_m(X) \le \iTC_{m+1}(X)$.
\end{proposition}

\begin{proof}
    Let $\iTC_{m+1}(X) = n-1$. Then there exists an $n$-intertwined $(m+1)$-navigation algorithm $s:X^{m+1} \to RP(\B_n(X))$. Let us fix some $x_0 \in X$ and define a map $J:X^m \to X^{m+1}$ such that
    \[
    J(x_1,\ldots,x_m) = (x_0,x_1,\ldots,x_m)
    \]
    for all $\ov{x} = (x_1,\ldots,x_m) \in X^m$. Let $sJ(\ov{x}) = \phi$. Then $\phi((i-1)/m) = \delta_{x_{i-1}}$ for all $2 \le i \le m+1$ and $\phi(0) = \delta_{x_0}$. In other words, $\phi(i/m) = \delta_{x_i}$ for all $1 \le i \le m$. Let $H:I \to I$ be the map
    \[
    H(s) = \frac{(m-1)s+1}{m}.
    \]
    Then for $\phi' = \phi H: I \to \B_n(X)$, we get $\phi'((i-1)/(m-1)) = \delta_{x_i}$ for all $1 \le i \le m$. Hence, we define a continuous map $K: X^m \to P(\B_n(X))$ such that $K(\ov{x}) = \phi'$. If $\mu = \sum \lambda_\alpha \alpha \in \B_n(P(X))$ is a resolver of $\phi$, then for $\alpha' = \alpha H$, the measure $\nu = \sum \lambda_\alpha \alpha' \in \B_n(P(X))$ is a resolver of $\phi'$. Hence, $K$ defines an $n$-intertwined $m$-navigation algorithm on $X$. Therefore, $\iTC_m(X) \le n-1$.
\end{proof}

Just like in the classical setting where the equality $\cat(X) = \TC(X)$ holds~\cite[Lemma 8.2]{Far2} for path-connected topological groups $X$, in the distributional setting also, the equality $\dcat(X) = \dTC(X)$ holds, see~\cite[Theorem 3.15]{DJ}. We now prove that such an equality holds for the intertwining invariants as well.

\begin{proposition}\label{topgrp}
    If $X$ is a path-connected topological group, then for all $m \ge 1$, we have $\iTC_{m+1}(X) \le \icat(X^m)$.
\end{proposition}

\begin{proof}
    Let $\icat(X^m) = n-1$ and $\ov{e} = (e,\ldots,e)\in X^m$ be the basepoint, where $e$ is the identity. This can be chosen as the basepoint since $X$ is path-connected. There exists a map $H: X^m \to RP_0(\B_n(X^m))$ satisfying $H(\ov{x})(0) = \delta_{\ov{x}}$ and $H(\ov{x})(1) = \delta_{\ov{e}}$ for each $\ov{x} \in X^m$. Since $X$ is a topological group, the map $J_m:X^{m+1} \to X^m$ defined as
    $$
    J_m\left(x_1,x_2,\ldots,x_{m+1}\right)= \left(x_2^{-1}x_1,x_3^{-1}x_2,\ldots,x_{m+1}^{-1}x_m\right),
    $$
    is continuous. Here, $x_i^{-1} \in X$ denotes the group-theoretic inverse of $x_i \in X$ and $x_{i+1}^{-1}x_i \in X$ is the group-theoretic product under the product map $\tau:X \times X \to X$. Let $HJ_m(x_1,\ldots,x_{m+1}) = \phi \in RP_0(\B_n(X^m))$. Using projections $\proj_i:X^{m} \to X$, we get induced maps $(\proj_i)_*:\B_n(X^{m}) \to \B_n(X)$ and paths $\phi_i = (\proj_i)_* \phi$ in $P_0(\B_n(X))$ that start at $\delta_{x_{i+1}^{-1}x_i}$ and end at $\delta_{e}$ for each $1 \le i \le m$. The map $\tau$ induces by functoriality of $\B_n$ a map $\mathcal{M}:X \times \B_n(X) \to \B_n(X)$ defined as
    \[
    \mathcal{M}\left(x,\sum b_y y \right)= \sum b_y \hspace{0.6mm}xy.
    \]
    For each $1\le i \le m$, we define a path $x_{i+1}\phi_i \in P(\B_n(X))$ such that for all $t \in I$,
    \[
    \left(x_{i+1}\phi_i\right)(t) = \mathcal{M}\left(x_{i+1},\phi_i(t)\right).
    \]
    Clearly, $(x_{i+1}\phi_i)(0) = \delta_{x_i}$ and $(x_{i+1}\phi_i)(1) = \delta_{x_{i+1}}$. Finally, we define a mapping $s:X^{m+1} \to P(\B_n(X))$ where, as usual,
    $$
    s(x_1,\ldots,x_{m+1}) = x_2\phi_1 \star x_3\phi_2 \star \cdots \star x_{m+1}\phi_m.
    $$
    Here, we are using the map $\theta_{m}:T_m(\B_n(X)) \to P(\B_n(X))$ from Section~\ref{pred} with $a_{i} = 1/t_{i+1} = m/i$ for all $1 \le i \le m-1$. By definition of $s$, we have that $s(x_1,\ldots,x_{m+1})\in P(\delta_{x_1},\ldots,\delta_{x_{m+1}})$. We still need to check if $s(x_1,\ldots,x_{m+1})$ is resolvable. Let $\mu = \sum \lambda_\alpha \alpha \in \B_n(P_0(X^m))$ be a resolver of $\phi \in RP_0(\B_n(X^m))$ and $\alpha_i=\proj_i\alpha$. Then, $\mu_i = \sum \lambda_\alpha \alpha_i \in \B_n(P_0(X))$ becomes a resolver of $\phi_i$ for each $i$. Let $x_{i+1}\alpha_i \in P(X)$ be given by $(x_{i+1}\alpha_i)(t) = x_{i+1}\alpha_i(t)$ for all $t \in I$. As before, we define a path 
    $$
    \wh{\alpha} = x_2\alpha_1 \star x_3\alpha_2 \star \cdots \star x_{m+1}\alpha_m.
    $$
    Then $\nu = \sum \lambda_\alpha \wh{\alpha} \in \B_n(P(X))$ is a resolver of $s(x_1,\ldots,x_{m+1})$. Hence, $s$ defines an $n$-intertwined $(m+1)$-navigation algorithm. So, $\iTC_{m+1}(X) \le n-1$.
\end{proof}

\begin{corollary}\label{usetopgrp}
    If $X$ is a path-connected topological group, then $\iTC(X) = \icat(X)$.
\end{corollary}
\begin{proof}
    From Proposition~\ref{topgrp}, we get $\iTC(X) \le \icat(X)$ by taking $m = 1$. The reverse inequality follows from Proposition~\ref{ine5}.
\end{proof}

\section{For Higman's group $\H$}\label{higmansecction} For each $m \ge 2$, we are interested in finding spaces $X$ (possibly depending on the choice of $m$) for which $\iTC_m(X) < \dTC_m(X)$. For $m=2$, such an example has been found in~\cite[Proposition 6.4]{Dr}. To obtain examples for each $m \ge 2$, we proceed as follows.

Let $\Gamma$ be a discrete group and let $B\Gamma$ denote the homotopy class of its classifying spaces. Since $\iTC_m$ and $\icat$ are homotopy invariants of spaces by Corollary~\ref{hinv} and Proposition~\ref{simple} (3), respectively, we can define $\iTC_m(\Gamma): = \iTC_m(B \Gamma)$ and $\icat(\Gamma): = \icat(B \Gamma)$. Let $\H$ denote Higman's group \cite{Hi} which is presented as follows.
$$
\H = \langle x,y,z,w \mid xyx^{-1}y^{-2},yzy^{-1}z^{-2}, zwz^{-1}w^{-2},wxw^{-1}x^{-2} \rangle.
$$
It is well-known that $\H$ is a torsion-free discrete group of cohomological dimension $2$ and it is acyclic, i.e., the (co)homology groups of $B\H$ vanish in positive degrees. Furthermore, $B\H$ is a CW complex of finite type. The $m^{th}$ sequential topological complexity of $\H$ was computed in~\cite[Theorem 2.2]{FO} as $2m$ for each $m \ge 2$ (also see~\cite[Theorem 4.1]{GLO2} for the special case $m = 2$).

In~\cite{Dr}, Dranishnikov showed that $\iTC(\H) = 1$. We use his technique to obtain the generalized result that $\iTC_m(\H) = 1$ for all $m \ge 2$. We first prepare as follows.

Let us fix some $m \ge 2$. For each $1 \le j \le m$, consider the polynomial 
$$
z_j(t) = \prod_{i=1,i\ne j}^m (t-t_i) = \prod_{i=1,i\ne j}^m \left(t-\frac{i-1}{m-1}\right)
$$
in the variable $t \in I$. For any fixed $j$, the only roots of $z_j$ are $t_i$ for all $i \ne j$. So, clearly, $z_j(t_j) \ne 0$ and $z_j(t_i) = 0$ for all $i \ne j$. Let 
$$
y_j(t) = \frac{z_j(t)}{z_j(t_j)}.
$$
Then $y_j(t_j) = 1$ and $y_j(t_i) = 0$ for all $i \ne j$. Next, for each $1 \le j \le m$, define 
$$
w_j(t) = \frac{\left(y_j(t)\right)^2}{\sum_{i=1}^m \left(y_i(t)\right)^2}.
$$
By definition, $\sum_{i=1}^m \left(y_j(t)\right)^2 \ne 0$ for all $t$, so $w_j(t)$ is well-defined and continuous. We observe the following. 
 \vspace{-2mm}
\begin{itemize}
        \itemsep-0.25em 
    \item For each $j$, $w_j(t) \in [0,1]$ for all $t \in I$.
    \item For each $j$, $w_j(t_j) = 1$ and $w_j(t_i) = 0$ for all $i \ne j$. 
    \item $\sum_{i=1}^m w_i(t) = 1$ for all $t \in I$.
\end{itemize}
Let $\ast^m X$ denote the iterated join of $X$. Let us denote the weak topology on $\ast^m X$ (in the sense of~\cite{KK}) by $\mathcal{T}_1$ and the strong topology (in the sense of~\cite{Mi}) by $\mathcal{T}_2$. Due to~\cite{Rut} (see also~\cite[Section 1]{FG}), the identity map $\mathscr{I}:(\ast^m X,\mathcal{T}_1)\to(\ast^m X,\mathcal{T}_2)$ is a homotopy equivalence. In fact, it is a homeomorphism when $X$ is compact. We define a map $q_m':X^m\times I \to (\ast^m X,\mathcal{T}_2)$ as
$$
q_m'\left(\left(x_1,x_2,\ldots,x_m\right),t\right) = \sum_{i=1}^{m} w_i(t)\hspace{0.5mm}x_i.
$$
Since each $w_i$ is continuous, the map $q_m'$ is continuous as well, see~\cite[Page 1]{Mi}. Hence, the composition $q_m=\mathscr{I}^{-1} q_m':X^m \times I \to (\ast^m X,\mathcal{T}_1)$ is continuous, where $\mathscr{I}^{-1}$ denotes a homotopy inverse of $\mathscr{I}$. In what follows, we only consider the weak topology $\mathcal{T}_1$ on joins. The map $q_m$ gives a continuous map $Q_m: X^m \to P(\ast^m X)$ such that for all $\ov{x} \in X^m$, we have $Q_m(\ov{x})(t_i)= x_i$ for each $1 \le i \le m$. 

We now extend the proof of the equality $\iTC(\H) = 1$ from~\cite{Dr} as follows.

\begin{proposition}\label{higman}
For Higman's group $\H$, $\iTC_m(\H) = 1$ for all $m \ge 2$.
\end{proposition}

\begin{proof}
    Let us fix some $m$. Since $\H$ is not contractible, we only need to show that $\iTC_m(\H) \le 1$. As noted above, $H_i(B\H) = 0$ for all $i > 0$. Let $\{\text{pt}\}$ denote a one-point space. Since, $H_i(B\H) = 0 = H_i(\{\text{pt}\})$ for each $ i > 0$, Dold's Theorem~\cite{Do} implies that $\wt{H}_i(SP^2(B\H)) = \wt{H}_i(SP^2(\{\text{pt}\})) = 0$ for each $ i > 0$. Hence, $SP^2(B\H)$ is acyclic. Furthermore, $\pi_1(SP^2(B\H)) = \wt{H}_1(B\H) = 0$. Then by Whitehead's theorem and Hurewicz's theorem~\cite{Ha}, $SP^2(B\H)$ is a contractible CW complex and hence an absolute extensor. We note that the disjoint union $\sqcup_{i=1}^m B\H$ is closed in the iterated join $\ast^m B\H$. So, if $\pa:B\H \to SP^2(B\H)$ denotes the diagonal embedding given by $\pa(x) = 2x$, then there exists an extension
\begin{equation}\label{zzzz}
    \begin{tikzcd}
    \bigsqcup_{i=1}^m B\H \arrow{r}{\bigsqcup_{i=1}^m \pa} \arrow[hookrightarrow]{d}
    & 
    SP^2(B\H)
    \\ 
    \ast^m B\H \arrow[ur, dashrightarrow]
    &
\end{tikzcd}
    \end{equation}
    of the map $\sqcup_{i=1}^m \pa$. This produces a map $f:P(\ast^m B\H) \to P(SP^2(B\H))$ of paths. The map $\phi:SP^2(B\H) \to \B_2(B\H)$ defined as
    $$
    \phi(x+y) = \frac{1}{2}x+\frac{1}{2}y
    $$
    is an embedding
    that produces a map $\phi':P(SP^2(B\H)) \to P(\B_2(B\H))$. Finally, consider the composition $\psi = \phi' f Q_m :(B\H)^m \to P(\B_2(B\H))$, where the map $Q_m:(B\H)^m \to P(\ast^m B\H)$ is defined above. By our construction of maps, for any $\ov{x} \in (B\H)^m$, we have that $\psi(\ov{x}) \in P(\delta_{x_1},\ldots, \delta_{x_m})$. We still need to check if $\psi(\ov{x})$ is resolvable. See that each path $g \in P(SP^2(B\H))$ is resolvable in the sense of the equivalent definition mentioned in Lemma~\ref{sames} upon replacing the role of the functor $\B_2$ by the functor $SP^2$ (see also the proof of~\cite[Proposition 6.4]{Dr}). Therefore, its image $\phi'(g) \in P(\B_2(B\H))$ is resolvable in the sense of Definition~\ref{resolve}. Thus, $\psi$ is a $2$-intertwined $m$-navigation algorithm. So, $\iTC_m(B\H) = \iTC_m(\H) = 1$.
\end{proof}

We note that in contrast, the exact values $\dTC_m(\H)$ are not known.

\begin{remark}
    Since $\H$ is torsion-free, it follows from~\cite[Theorem 5.3]{Dr} (see also~\cite[Theorem 7.4]{KW} for a version of the Eilenberg--Ganea theorem~\cite{EG} for the \textit{analog} invariants) that $\dcat(\H) = 2$. Thus, from  Proposition~\ref{higman} and~\cite[Proposition 3.5]{Ja}, we get for each $m \ge 2$ that 
$$
\iTC_m(\H) = 1 < 2(m-1) = 
\dcat(\H^{m-1}) \le \dTC_m(\H).
$$
So, we have shown that for each $m \ge 2$, the first inequality in Proposition~\ref{ine1} can be strict. Hence, in general, the notions of $\iTC_m$ and $\dTC_m$ are different! Furthermore, as noted in~\cite[Corollary 6.5]{Dr}, we get $\icat(\H) = 1$ due to Propositions~\ref{higman} and~\ref{ine5} (1). So, while an analog of the Eilenberg--Ganea theorem holds (only) for torsion-free groups in the distributional setting, such an analog is \textit{not} possible in the intertwining setting, even for torsion-free groups. 
\end{remark}

For a fixed $X$, the non-decreasing sequence $\{\iTC_m(X)\}_{m \ge 2}$ can, in fact, be constant due to Proposition~\ref{higman}. Besides $X = B\H$, another example for which the sequence is again constant at $1$ is $X = \R P^{\infty}$. This follows from Corollary~\ref{vvobv} and~\cite[Section 8.A]{Ja} in view of~\cite[Corollary 7.3]{KW}.

\begin{remark}
In the case of classical $\TC_m$, if $X$ is a non-contractible finite CW complex, then $\TC_m(X) \ge m-1$, see~\cite[Proposition 3.5]{Rud}. However, due to Proposition~\ref{higman}, such a statement does not hold in the case of $\iTC_m$ for any $m \ge 3$.
\end{remark}

The classifying space $B\H$ of Higman's group is an example of a finite-dimensional aspherical CW complex for which $\icat$ and $\iTC_m$ values agree for all $m$, even though it is neither a topological group nor a (product of) co-$H$-space(s). Such an example is not known in the distributional setting or the classical setting.

\section{Some simple characterizations}\label{Characterizations}

Let $\pi_m:RP(\B_n(X)) \to (\B_n(X))^m$ be the evaluation map 
$$
\pi_m(\phi) = \left(\phi(t_1),\phi(t_2),\ldots,\phi(t_m)\right),
$$
where $t_i = (i-1)/(m-1)$ for all $1 \le i \le m$. Recall the inclusion $\omega_n:X \to \B_n(X)$ defined as $\omega_n(x) = \delta_x$. Consider the following pullback diagram, 
\begin{equation}\label{one1}
    \begin{tikzcd}
\U_{n,m} \arrow{r}{\eta_m} \arrow[swap]{d}{q_m} 
& 
RP(\B_n(X)) \arrow{d}{\pi_m}
\\
X^m  \arrow{r}{\omega_n^m}
& 
(\B_n(X))^m
\end{tikzcd}
\end{equation}
where $\omega_n^m(\ov{x}) = (\omega_n(x_1),\omega_n(x_2),\ldots,\omega_n(x_m)) = (\delta_{x_1},\ldots,\delta_{x_m})$.

\begin{proposition}\label{fake1}
    $\iTC_m(X) < n$ if and only if there exists a section to the map $q_m:\U_{n,m} \to X^m$.
\end{proposition}

\begin{proof}
    First, let $\iTC_m(X) < n$. Then there exists an $n$-intertwined $m$-navigation algorithm $s_m: X^m \to RP(\B_n(X))$. Note that $\pi_m \hspace{0.4mm} s_m(\ov{x}) = (\delta_{x_1},\ldots,\delta_{x_m}) = \omega_n^m(\ov{x})$. So, by the universal property of the above pullback, there exists $K:X^m \to \U_{n,m}$ such that $q_m K = \mathbbm{1}_{X^m}$. So, $q_m$ admits a section. Conversely, let $K: X^m \to \U_{n,m}$ be a section of $q_m$, i.e., $q_m K = \mathbbm{1}_{X^m}$. Note that 
$$
\pi_m(\eta_m K) = (\pi_m \eta_m)K = (\omega_n^m q_m)K = \omega_n^m(q_m K) = \omega_n^m.
$$
Hence, $(\eta_m K)(\ov{x}) \in P(\delta_{x_1},\ldots,\delta_{x_m})$. So, $\eta_m K$ defines an $n$-intertwined $m$-navigation algorithm on $X$. Therefore, $\iTC_m(X) \le n-1$.
\end{proof}

Similarly, if we consider the evaluation map $p_0:RP_0(\B_n(X)) \to \B_n(X)$ defined as $p_0(\phi) = \phi(0)$ and let $q:\V_n \to X$ be the pullback of $p_0$ along $\omega_n$ for the pullback space $\V_n$, then the following statement follows immediately.

\begin{proposition}\label{fake2}
    $\icat(X) < n$ if and only if there exists a section to the map $q:\V_n \to X$.
\end{proposition}

\begin{remark}
    We note that the evaluation maps $\wh{\pi_m}: P(\B_n(X)) \to (\B_n(X))^m$ and $\wh{p_0}:P_0(\B_n(X))\to \B_n(X)$ are fibrations. So, due to Corollary~\ref{useonce}, they are the fibrational substitutes of the maps $\pi_m$ and $p_0$, respectively. In the above setting, the maps $\pi_m$ and $p_0$ cannot be replaced by their fibrational substitutes $\wh{\pi_m}$ and $\wh{p_0}$ because then, the reverse implications will not hold in the above propositions.
\end{remark}

Let $RP_1(\B_n(X)) = \{f \in RP(\B_n(X)) \mid f(0) = \delta_{x_0}\}$ for some fixed $x_0 \in X$. We define an evaluation map $\xi_m:RP_1(\B_n(X)) \to (\B_n(X))^m$ as
$$
\xi_m (\phi)= \left(\phi\left(\frac{1}{m}\right), \phi\left(\frac{2}{m}\right), \ldots, \phi\left(\frac{m-1}{m}\right), \phi\left(1\right) \right).
$$
As before, we consider the following pullback diagram.
\begin{equation}\label{three3}
    \begin{tikzcd}
\W_{n,m} \arrow{r}{\rho_m} \arrow[swap]{d}{r_m} 
& 
RP_1(\B_n(X)) \arrow{d}{\xi_m}
\\
X^m  \arrow{r}{\omega_n^m}
& 
(\B_n(X))^m
\end{tikzcd}
\end{equation}

The following statement may be compared with~\cite[Proposition 5.7]{Ja}.

\begin{proposition}
    For any $m \ge 1$, if $\icat(X^m) < n$, then there exists a section to $r_m:\W_{n,m} \to X^m$.
\end{proposition}

\begin{proof}
If $\icat(X^m) < n$, then there exists a continuous map $H:X^m \to RP_0(\B_n(X^m))$ such that for some fixed basepoint $\ov{x_0} = (x_1^0,\ldots,x_m^0) \in X^m$, $H(\ov{x})(0) = \delta_{\ov{x}}$ and $H(\ov{x})(1) = \delta_{\ov{x_0}}$ for all $\ov{x} \in X^m$. Let $H(\ov{x}) = \phi$. For projections $\proj_i:X^m \to X$, let $\phi_i = (\proj_i)_*\ov{\phi} \in P(\B_n(X))$ be paths from $\delta_{x_i^0}$ to $\delta_{x_i}$. For each $i$, let $\gamma_i \in P(X)$ be a path from $x_i^0$ to $x_{i+1}^0$ and $\gamma_i^* = \omega_n\gamma_i \in P(\B_n(X))$. Finally, define a map $K: X^m \to P(\B_n(X))$ such that
$$
K(\ov{x}) = \phi_1 \star \left(\ov{\phi}_1 \cdot \gamma_1^* \cdot \phi_2\right) \star \left(\ov{\phi}_2 \cdot \gamma_2^* \cdot \phi_3 \right) \star \cdots \star \left(\ov{\phi}_{m-1} \cdot \gamma_{m-1}^* \cdot \phi_m\right),
$$
where we use the map $\theta_m:T_m(\B_n(X)) \to P(\B_n(X))$ in Section~\ref{pred} with $a_i = i/m$ for all $1 \le i \le m-1$. The continuity of $K$ follows from Lemma~\ref{iscont}. For all $\ov{x} \in X^m$, we have $K(\ov{x})(0) = \delta_{x_1^0}$ and $K(\ov{x})(i/m) = \delta_{x_i}$ for all $1 \le i \le m$. Now, we check the resolvability of $K(\ov{x})$. Let $\mu = \sum \lambda_\alpha \alpha \in \B_n(P(X))$ be a resolver of $\phi \in RP_0(\B_n(X))$. Letting $\alpha_i = \proj_i \ov{\alpha}$, we can define a path 
$$
\wt{\alpha} = \alpha_1 \star \left(\ov{\alpha}_1 \cdot \gamma_1 \cdot \alpha_2\right) \star \left(\ov{\alpha}_2 \cdot \gamma_2 \cdot \alpha_3 \right) \star \cdots \star \left(\ov{\alpha}_{m-1} \cdot \gamma_{m-1} \cdot \alpha_m\right).
$$
Then the measure $\nu = \sum \lambda_\alpha \wt{\alpha} \in \B_n(P(X))$ is a resolver of $K(\ov{x})$. Hence, the image of $K$ is in $RP_1(\B_n(X))$. Clearly, $\xi_m K= \omega_n^m$. So, by the universal property of the pullback in Diagram~\ref{three3}, we obtain a section $s:X^m \to \W_{n,m}$ to the map $r_m$.
\end{proof}

\section{Cohomological lower bounds}\label{Bounds}
In the case of intertwining motion, there is a lot of liberty with the (pieces of the) system while navigating between given positions (see Section~\ref{intmot}). Indeed, the intertwining of the strings of a resolvable path can happen in various ways (see Section~\ref{Paths2}). Due to this liberty, for most configuration spaces, very few rules may be needed to plan an intertwining motion in the space. For any non-contractible configuration space, we know that at least two rules are needed. So, to enrich the theory of the intertwining invariants, it becomes important to show that there exist spaces for which at least three rules are required to plan such a motion.

The aim of this section is to obtain lower bounds for $\icat(X)$ and $\iTC_m(X)$. As in the case of $\dcat(X)$ and $\dTC_m(X)$, we aim to find these in terms of the cohomology of some functor of the space $X$ and its products $X^m$. As a starting point, we check if the Dirac inclusion $\omega_n:X \hookrightarrow \B_n(X)$ can be of help. For a coefficient ring $R$, this induces a homomorphism $\omega_n^*:H^*(\B_n(X);R) \to H^*(X;R)$. 

\begin{remark}\label{useless1}
    If $\alpha \in H^k(\B_n(X);R)$ for some $k \ge 1$ and ring $R$ such that $\omega_n^*(\alpha) \ne 0$, then $\icat(X) \ge n$. To see this, suppose $\icat(X) < n$. Then there exists a map $H : X \to RP_0(\B_n(X))$ such that $H(x)(0) = \delta_{x}$ and $H(x)(1) = \delta_{x_0}$ for each $x \in X$. The map $K:X \times I \to \B_n(X)$ given by
    $K(x,t) = H(x)(t)$ defines a null-homotopy. In particular, this means $\omega_n^*(\alpha) = 0$. 
    Similarly, we can obtain a formal lower bound to $\iTC_m(X)$ for each $m\ge 2$. Recall Diagram~\ref{one1} and let $\Delta_n':\B_n(X) \to (\B_n(X))^m$ be the diagonal map. It is easy to show that $\iTC_m(X) \ge n$ if $\alpha \in H^k((\B_n(X))^m;R)$ for some $k \ge 1$ and ring $R$ such that $(\Delta_n')^*(\alpha) = 0$ and $(\omega_n^m)^*(\alpha) \ne 0$.
\end{remark}

\begin{lemma}
    For any ring $R$ and $n\ge 2$, the map $\omega_n^*:H^*(\B_n(X);R)\to H^*(X;R)$ is trivial.
\end{lemma}

\begin{proof}
    For a fixed basepoint $x_0\in X$, let $CX$ denote the reduced cone of $X$ given by $CX=
    (X\times[0,1])/(X\times \{0\}\cup\{x_0\}\times[0,1])$. Let us denote the basepoint of $CX$ by $[x_0]$, which is the equivalence class of $(x_0,0)$. The other points are formal linear combinations $xt+(1-t)x_0$ for $x\ne x_0$ and $t>0$. Consider the map $F:CX\to\B_2(X)$ defined as follows: $F(xt+(1-t)x_0)=t \delta_x + (1-t)\delta_{x_0}$ for $x\ne x_0$ and $t>0$, and $F([x_0]) = \delta_{x_0}$. Clearly, $F$ is continuous and thus, we have the map $X\hookrightarrow CX\xrightarrow{F}\B_2(X)\hookrightarrow\B_n(X)$ for any $n \ge 2$. Hence, our map $\omega_n:X \to\B_n(X)$ defined as $\omega_n(x) = \delta_x$ factors through $CX$. But $CX$ is contractible to $[x_0]$. Thus, it follows that $\omega_n^*=0$ with respect to all coefficients.
\end{proof}

So, the lower bounds mentioned in Remark~\ref{useless1} are not useful. In particular, we cannot obtain $2$ as a lower bound to $\icat(X)$ or $\iTC_m(X)$ using Remark~\ref{useless1}.

The next idea is to proceed using the symmetric products as in~\cite{DJ},~\cite{Ja}. But, we can't directly adopt the approach of~\cite[Lemma 4.5]{DJ} because the resolvers of a resolvable path can have wildly different characteristics --- see Section~\ref{Paths2} for various examples. So, in general, it is difficult to characterize resolvable paths.

\subsection{Preparation}\label{prepo}
In this subsection, we prepare to provide various conditions in which $2$ could be the lower bound to the intertwining invariants of spaces.

Let us fix some $\mu = af+bg\in \B_2(P(X))$ where $a,b > 0$ and $f \ne g$ as paths in $P(X)$. Then we can see $\mu$ as an element of $\text{Sym}(*^2 P(X))$ and write it as $[af,bg]$. 

Consider the natural map $\Psi:\text{Sym}(*^2 P(X)) \to P(\text{Sym}(*^2 X))$ defined such that for any $[s_1\phi_1,s_2\phi_2]\in\text{Sym}(\ast^2 P(X))$, we have $\Psi([s_1\phi_1,s_2\phi_2])(t) = [s_1\phi_1(t),s_2\phi_2(t)]$ for all $t\in I$. For $\mu$ as above, let us denote $\Psi(\mu) \in P(\text{Sym}(\ast^2 X))$ by $\wt{\mu}$. Let $Y(\mu) = \wt{\mu}(I)=\{[af(t),bg(t)] \mid t \in I\} \subset \text{Sym}(*^2 X)$. We recall from Section~\ref{delicate} the continuous map $\mathcal{I}:(\B_2(X),\mathscr{T}_1) \to (\B_2(X),\mathscr{T}_2)$. Let $q(\mu):Y(\mu) \to (\B_2(X),\mathscr{T}_1)$ denote the restriction of the quotient map $q:\text{Sym}(*^2 X) \to (\B_2(X),\mathscr{T}_1)$ to $Y(\mu)$.

\begin{lemma}\label{embed}
    The map $\wt{q}(\mu) = \mathcal{I}q(\mu):Y(\mu) \to (\B_2(X),\mathscr{T}_2)$ is an embedding.
\end{lemma}

\begin{proof}
    Note that the $Y(\mu)$ is compact and $(\B_2(X),\mathscr{T}_2)$ is Hausdorff. Hence, $\wt{q}(\mu)$ is a closed map. Now, we show that $q(\mu)$ is injective. Let $s,t \in I$ be such that 
    \[
    q(\mu)\left([af(t),bg(t)]\right) = q(\mu)\left([af(s),bg(s)]\right)=\nu.
    \]
    If $|\supp(\nu)|=1$, then $f(t) = g(t) = f(s) = g(s)$, which obviously implies that $[af(t),bg(t)]=[af(s),bg(s)]$. On the other hand, if $|\supp(\nu)|=2$, then $f(t) \ne g(t)$ and $f(s) \ne g(s)$. If $a > b$, then we must have $f(t) = f(s)$ and $g(t) = g(s)$, and hence, $[af(t),bg(t)]=[af(s),bg(s)]$. If $a=b$, then we get $[f(t),g(t)] = [f(s),g(s)]$ in $SP^2(X)$, which implies $[af(t),bg(t)]=[af(s),bg(s)]$ in $\text{Sym}(\ast^2X)$. This proves that $q(\mu)$ is injective. Since $\mathcal{I}$ is injective, the map $\wt{q}(\mu)$ is also injective. Finally, since $\wt{q}(\mu)$ is closed, it follows that it is an embedding.
\end{proof}

We denote $\wt{q}(\mu)(Y(\mu)) \subset \B_2(X)$ by $Z(\mu)$. Lemma~\ref{embed} gives a continuous map $\wt{q}(\mu)^{-1}:Z(\mu) \to Y(\mu)$ such that $\wt{q}(\mu)^{-1}(af(t)+bg(t))=[af(t),bg(t)]$ for all $t\in I$. Corresponding to $\mu$ fixed above, define $W(\mu) = \{(af(t),bg(t))\mid t\in I\}\subset \ast^2X$.
Let us recall from Section~\ref{higmansecction} the weak topology $\mathcal{T}_1$ and the strong topology $\mathcal{T}_2$ on the iterated join $\ast^2 X$ and the identity map $\mathscr{I}:(\ast^2X,\mathcal{T}_1)\to(\ast^2X,\mathcal{T}_2)$, which is a homotopy equivalence,~\cite{Rut},~\cite{FG}. Let $\mathscr{I}'$ be the restriction of $\mathscr{I}$ to $W(\mu)$. Let us define a map $r(\mu):(W(\mu),\mathcal{T}_2)\to X^2$ by $r(\mu)(af(t),bg(t))=(f(t),g(t))$. By definition, $r(\mu)$ is continuous, see~\cite[Page 1]{Mi}. So, the composition $r(\mu)\mathscr{I}':(W(\mu),\mathcal{T}_1)\to X^2$ is continuous. Since it is also $\Z_2$-equivariant, we obtain a continuous mapping $\wt{r}(\mu):Y(\mu)\to SP^2(X)$ defined as $\wt{r}(\mu)([af(t),bg(t)])=[f(t),g(t)]=f(t)+g(t)$ due to our convention on $SP^2(X)$ from Section~\ref{pre2}. Thus, we have a continuous map $\wt{r}(\mu)\wt{q}(\mu)^{-1}:Z(\mu) \to SP^2(X)$ such that for all $t \in I$,
\[
\wt{r}\left(\mu\right)\hspace{0.3mm}\wt{q}\left(\mu\right)^{-1}\left(af(t) + bg(t)\right)= f(t)+g(t).
\]

\begin{remark}\label{indep}
    Let $\mu = af+bg$ and $\nu = a'f'+b'g'$ be two resolvers of $\phi \in RP(\B_2(X))$ with $a,b,a',b'>0$, $f\ne g$, and $f'\ne g'$. We claim that $\wt{\mu}(t)= \wt{\nu}(t)$ for all $t \in I$. 
    Since $f\ne g$, there exists $s\in(0,1)$ such that $f(s)\ne g(s)$. Since $af(t)+bg(t)=a'f'(t)+b'g'(t)$ for each $t\in I$, 
    $f'(s)\ne g'(s)$. First, let $a=b=\frac{1}{2}$. Then due to the above equality, $a'=b'=\frac{1}{2}$. So, $[f(t),g(t)]=[f'(t),g'(t)]$ holds for all $t\in[0,1]$ and thus, $\wt{\mu}(t)=\wt{\nu}(t)$. Next, we let $a>b$. Then because $af(s)+bg(s)=a'f'(s)+b'g'(s)$ for $s\in (0,1)$ as above, either $a=a'$ and $b=b'$, or $a=b'$ and $b=a'$. In the first case, $af(t)=a'f'(t)$ and $bg(t)=b'g'(t)$ for all $t\in[0,1]$. In the second case, $af(t)=b'g'(t)$ and $bg(t)=a'f'(t)$. So, in any case, $\wt{\mu}(t)= \wt{\nu}(t)$ for all $t \in I$. Hence, $Y(\mu) = Y(\nu)$, and as maps, $\wt{q}(\mu) = \wt{q}(\nu)$ and $\wt{r}(\mu)\wt{q}(\mu)^{-1} = \wt{r}(\nu)\wt{q}(\nu)^{-1}$.
\end{remark}

 As the following two examples show, when $n \ge 3$, these nice properties are not satisfied by the measures in $\B_n(P(X))$ that have supports of maximum size.

\begin{example}\label{notgeneral}
    Let $\mu = \frac{1}{2}\alpha+\frac{1}{4}(\gamma,\beta) + \frac{1}{4}\gamma \in \B_3(P(X))$, in the notations of Section~\ref{Paths2} where $\beta(t_1)=\gamma(t_1)$ for some $t_1\in(0,1)$. Clearly, $\wt{q}(\mu)(t) = \frac{1}{2}\alpha(t) + \frac{1}{2}\gamma(t)$ for all $t \in (0,t_1)$. Let $t,s \in (0,t_1)$ such that $\alpha(t) = \gamma(s) \ne \gamma(t)=\alpha(s)$. Let us define the map $\wt{q}(\mu):Y(\mu)\to(\B_3(X),\mathscr{T}_2)$ as in Lemma~\ref{embed}. Then 
    \[
    \wt{q}(\mu)\left[\frac{1}{2}\alpha(t),\frac{1}{4}\gamma(t),\frac{1}{4}\gamma(t)\right] = \frac{1}{2}\alpha(t) + \frac{1}{2}\gamma(t) =  \wt{q}(\mu)\left[\frac{1}{2}\alpha(s),\frac{1}{4}\gamma(s),\frac{1}{4}\gamma(s)\right] 
    \]
    Hence, $\wt{q}(\mu)$ is not injective and thus, it is not an embedding.
    \end{example}

    \begin{example}
        Consider a resolvable path $\phi:I \to \B_3(X)$ described as follows. At $t=0$, the path breaks into two strings of weights $\frac{1}{2}$ each. At $t=t_1$, the strings intertwine and break into three strings of weights $\frac{1}{2},\frac{1}{4}$, and $\frac{1}{4}$. Then they recombine at $t=1$. While the measure $\mu \in \B_3(P(X))$ defined in Example~\ref{notgeneral} is one resolver of $\phi$, another is $\nu = \frac{1}{2}(\gamma,\alpha)+\frac{1}{4}(\alpha,\beta) + \frac{1}{4}(\alpha,\gamma) \in \B_3(P(X))$. By the description of $\phi$, the paths $\alpha$ and $\gamma$ are not identical on $(0,t_1)$. Thus, there exists some $t \in (0,t_1)$ such that $\alpha(t) \ne \gamma(t)$. Then for this $t$, in the above notations, we have
        \[
        \wt{\mu}(t) = \left[\frac{1}{2}\alpha(t),\frac{1}{4}\gamma(t),\frac{1}{4}\gamma(t)\right] \ne  \left[\frac{1}{2}\gamma(t),\frac{1}{4}\alpha(t),\frac{1}{4}\alpha(t)\right] = \wt{\nu}(t).
        \]
        Therefore, the outputs are not \textit{resolver invariant} anymore. 
    \end{example}

\subsection{For $\icat$}
Recall the diagonal map $\pa_2:X \to SP^2(X)$ defined in Section~\ref{pre2} as $\pa_2(x)= [x,x]=2x$, which we called \textit{the inclusion} of $X$ into $SP^2(X)$.

\begin{lemma}\label{imp=2}
If $\icat(X) < 2$, then there exist sets $A_1$ and $A_2$ that cover $X$ such that the inclusion $A_i \to SP^2(X)$ is null-homotopic for $i = 1,2$. 
\end{lemma}

\begin{proof}
    If $\icat(X)<2$, then for a basepoint $x_0 \in X$, there exists a continuous map $H:X \to RP_0(\B_2(X))$ such that $H(x)(0) = \delta_x$ and $H(x)(1) = \delta_{x_0}$ for all $x \in X$. Define
    \[
    A_1 = \left\{ x \in X \hspace{1mm} \middle| \hspace{1mm} H(x) \text{ has a resolver of support } 1 \right\}.
    \]
Clearly, if $x \in A_1$, then any resolver of $H(x)$ is in $P(X)$ and thus, $H(x)(t) \in X$ for all $t \in I$. So, we define $K_1:A_1 \times I \to SP^2(X)$ as the composition
    \[
    K_1:(x,t) \mapsto H(x)(t) \xmapsto{\pa_2} 2H(x)(t).
    \]
    Let $A_2 = X \setminus A_1$. For each $x \in A_2$, we choose and fix a resolver of $H(x)$, say $\mu_x =a^x\phi_1^x+b^x\phi_2^x \in \B_2(P(X))$. Then $a^x,b^x>0$ and $\phi_1^x\ne\phi_2^x$, and in the notations of Section~\ref{prepo}, we have $H(x)(t) \in Z(\mu_x)\subset \B_2(X)$ for all $t \in I$. Let us define $f_x:I\to SP^2(X)$ as the composition
    \[
    f_x:t\mapsto H(x)(t)=a^x\phi_1^x(t) + b^x\phi_2^x(t) \xmapsto{\wt{r}(\mu_x)\wt{q}(\mu_x)^{-1}} \phi_1^x(t)+\phi_2^x(t).
    \]
    In view of Remark~\ref{indep}, the definition of $f_x$ is independent of the choice of $\mu_x$. Hence, $f_x$ is well-defined, and being a composition of continuous maps, it is continuous. Finally, we define a homotopy $K_2:A_2\times I \to SP^2(X)$ by the formula
        \[
    K_2(x,t) = f_x(t).
    \]
For each $i$ and all $x \in A_i$, $K_i(x,0) = 2x = \pa_2(x)$ and $K_i(x,1) = 2x_0$ by Remark~\ref{basic}. Hence, we conclude that the maps $K_i$ are the required null-homotopies. 
\end{proof}

We now proceed as in~\cite[Section 4.2]{DJ}. In Alexander--Spanier cohomology, we get the following statement.



\begin{theorem}\label{catbound}
 Suppose that $\alpha_{i}^{*} \in H^{k_{i}}\left(SP^{2}(X);R\right)$ for some ring $R$ and $k_{i} \geq 1$ for $i = 1,2$. Let  $\alpha_{i} \in H^{k_{i}}(X;R)$ be the image of $\alpha_{i}^{*}$ under the induced map $\pa_2^*$ such that  $\alpha_{1} \smile \alpha_{2} \neq 0$. Then $\icat(X) \geq 2.$
 \end{theorem}

The proof uses Lemma~\ref{imp=2} and is the same as that of~\cite[Theorem 4.7]{DJ} for the particular case $n=2$ (see also~\cite[Proposition 1.5]{CLOT}).

\begin{remark}
 We note that the reason for using Alexander--Spanier cohomology is that in the proof, we need to use the long exact sequence of the pairs $(SP^2(X), A_i)$. But here, the sets $A_i$ need not be open or closed in $SP^2(X)$. So, $(SP^2(X), A_i)$ may not be a good pair in the sense of~\cite{Ha} to use singular cohomology.
\end{remark}

\begin{corollary}\label{rationalcl}
    If $\max\{c\ell_{\R}(X),c\ell_{\Q}(X)\} \ge 2$ for a finite CW complex $X$, then $\icat(X) \ge 2$.
\end{corollary}
\begin{proof}
    Since the Alexander--Spanier cohomology groups are the same as the singular cohomology groups for locally finite CW complexes~\cite{Sp}, this statement follows directly from Theorem~\ref{catbound} and Proposition~\ref{useful}. We note that the finiteness hypothesis on the CW complex is needed to use Proposition~\ref{useful}.
\end{proof}

\subsection{For $\iTC$} We recall that a deformation of a subset $A\subset X$ to another subset $D\subset X$ is a continuous map $H:A\times I\to X$ such that $H|_{A\times\{0\}}:A\to X$ is the inclusion and $H(A\times\{1\})\subset D$.  

\begin{lemma}
If $\iTC(X) < 2$, then there exist sets $B_1$ and $B_2$ that cover $X \times X$ such that the set $B_i$ deforms to the diagonal $\Delta X$ inside $SP^2(X \times X)$ for $i = 1,2$. 
\end{lemma}

\begin{proof}
    If $\iTC(X)<2$, then there exists a continuous map $s:X \times X \to RP(\B_2(X))$ such that $s(x,y)(0) = \delta_x$ and $s(x,y)(1) = \delta_{y}$ for all $(x,y) \in X \times X$. Define
    \[
    B_1 = \left\{ (x,y) \in X \times X \hspace{1mm} \middle| \hspace{1mm} s(x,y) \text{ has a resolver of support } 1 \right\}.
    \]
   We note that if $(x,y)\in B_1$, then $s(x,y)(t)\in X$ for all $t$. Thus, let us define a map $K_1:B_1\times I\to SP^2(X\times X)$ as the composition
   \[
    K_1:(x,y,t) \mapsto (x,y,t,y) \mapsto (s(x,y)(t),y) \mapsto 2(s(x,y)(t),y).
    \]
    Let $B_2 = (X \times X) \setminus B_1$. For each $(x,y)\in B_2$, let $\mu_{(x,y)}=a^{(x,y)}\phi_1^{(x,y)}+b^{(x,y)}\phi_2^{(x,y)}$ be a fixed resolver of $s(x,y)$. Now, just like in the proof of Lemma~\ref{imp=2}, we define $f_{(x,y)}:I \to SP^2(X)$ as the composition
\[
f_{(x,y)}:t\mapsto s(x,y)(t) = a^{(x,y)}\phi_1^{(x,y)}(t) + b^{(x,y)}\phi_2^{(x,y)}(t)\mapsto \phi_1^{(x,y)}(t) +\phi_2^{(x,y)}(t)
\]
using the map $\wt{r}(\mu_{(x,y)})\wt{q}(\mu_{(x,y)})^{-1}$. Due to Remark~\ref{indep}, $f_{(x,y)}$ is independent of the choice of $\mu_{(x,y)}$. Thus, being a composition of continuous maps, it is continuous. Finally, we define a continuous map $K_2:B_2\times I\to SP^2(X\times X)$ as the composition
    \[
    K_2:(x,y,t)\mapsto (x,y,t,y)\mapsto (f_{(x,y)}(t),y)\mapsto \left(\phi_1^{(x,y)}(t),y\right)+\left(\phi_2^{(x,y)}(t),y\right).
    \]
    The second map used in the above composition is continuous due to the functoriality of $SP^2$. Note that $K_i(x,y,0) = 2(x,y)$ and $K_i(x,y,1) = 2(y,y)$ for $i=1,2$ in view of Remark~\ref{basic}. Thus, each $K_i$ defines the required deformation.
\end{proof}




The proof of the following theorem in Alexander--Spanier cohomology is the same as that of~\cite[Theorem 4.12]{DJ} for the case $n=2$ (see also~\cite[Theorem 7]{Far1}).

\begin{theorem}\label{boundtc}
Suppose that $\beta_{i}^{*} \in H^{k_i}\left(SP^{2}(X \times X);R\right)$ for some ring $R$ and $k_{i} \ge 1$ for $i = 1,2$. Let $\beta_i \in H^{k_i}(X \times X;R)$ be the image of $\beta_i^*$ under the homomorphism induced by the inclusion. If $\beta_i$ are zero-divisors such that
 $\beta_{1} \smile \beta_{2} \neq 0$, then $\iTC(X) \geq 2$.
\end{theorem}

\begin{corollary}\label{rationalzcl}
    If $\max\{zc\ell_{\R}(X),zc\ell_{\Q}(X)\} \ge 2$ for a finite CW complex $X$, then $\iTC(X) \ge 2$.
\end{corollary}

\begin{proof}
    Since the Alexander--Spanier cohomology groups are the same as the singular cohomology groups for locally finite CW complexes~\cite{Sp}, this statement follows directly from Theorem~\ref{boundtc} and Proposition~\ref{useful}.
\end{proof}

\subsection{For $\iTC_m$} To obtain $2$ as a lower bound to $\iTC_m(X)$ for all spaces $X$ and $m \ge 2$, we proceed in a way similar to~\cite[Section 4.B]{Ja}. However, we will have to make some major modifications.  

For a fixed $m \ge 2$, let $\zeta:P(SP^{2}(X)) \to (SP^2(X))^m$ be the fibration
$$
\zeta(\phi)= \left(\phi(t_1),\phi(t_2),\ldots,\phi(t_m)\right)
$$
for $t_i = (i-1)/(m-1)$. 
The diagonal map $\pa_2:X \to SP^{2}(X)$ gives the product $\pa_2^m:X^m \to (SP^{2}(X))^m$ defined as $\pa_2^m \left(x_1,\ldots,x_m\right)= (2 x_1, \ldots, 2 x_m)$. Let us consider the following pullback diagram.
\begin{equation}\label{four4}
\begin{tikzcd}[contains/.style = {draw=none,"\in" description,sloped}]
\mathcal{D}_{m} \arrow{r}{a} \arrow[swap]{d}{\sigma} 
& 
P(SP^{2}(X)) \arrow[swap]{d}{\zeta}
\\
X^m  \arrow{r}{\pa_2^m}
& 
(SP^{2}(X))^m
\end{tikzcd}
\end{equation}
We note that in this setting, one can replace the set $X^m$ on the left with any of its subset $C_i$ and obtain a pullback diagram as above with the corresponding pullback space $\mathcal{D}_{m,i}$ and the canonical projection fibration $\sigma_i:\mathcal{D}_{m,i} \to C_i$.

\begin{lemma}\label{second}
If $\iTC_m(X) < 2$, then there exist sets $C_1$ and $C_2$ that cover $X^m$ and over each of which the fibration $\sigma_i:\mathcal{D}_{m,i} \to C_i$ has a section.
\end{lemma}

\begin{proof}
We proceed as in Lemma~\ref{imp=2}. Since $\iTC_m(X) < 2$, there exists a $2$-intertwined $m$-navigation algorithm $s_m: X^m \to RP(\B_2(X))$. Define
    \[
    C_1 = \left\{ \ov{x} = (x_1,\ldots,x_m) \in X^m \hspace{1mm} \middle| \hspace{1mm} s_m(\ov{x}) \text{ has a resolver of support } 1 \right\}.
    \]
    We define the map $K_1:C_1 \times I \to SP^2(X)$ as follows:
    \[
    K_1:(\ov{x},t) \mapsto s_m(\ov{x})(t) \mapsto 2\hspace{0.3mm}s_m(\ov{x})(t). 
    \]
    Let $C_2 = X^m \setminus C_1$. For each $\ov{x}\in C_2$, let $\mu_{\ov{x}}=a^{\ov{x}}\phi_1^{\ov{x}}+b^{\ov{x}}\phi_2^{\ov{x}}$ be a resolver of $s_m(\ov{x})$. Note that $s_m(\ov{x})(t)\in Z(\mu_{\ov{x}})\subset \B_2(X)$ for each $t\in I$, in the notations of Section~\ref{prepo}. As before, we define the continuous path $f_{\ov{x}}:I\to SP^2(X)$ as follows:
    \[
    f_{\ov{x}}:t\mapsto s_m(\ov{x})(t)=a^{\ov{x}}\phi_1^{\ov{x}}(t)+b^{\ov{x}}\phi_2^{\ov{x}}(t) \xmapsto{\wt{r}(\mu_{\ov{x}})\wt{q}(\mu_{\ov{x}})^{-1}} \phi_1^{\ov{x}}(t)+ \phi_2^{\ov{x}}(t).
    \]
    This gives us a continuous map $K_2:C_2\times I\to SP^2(X)$ defined by the formula
    \[
    K_2(\ov{x},t)=f_{\ov{x}}(t).
    \]
    Therefore, for $i=1,2$, we obtain continuous maps $H_i:C_i \to P(SP^2(X))$ such that $\zeta H_i(\ov{x}) = (2x_1,\ldots,2x_m)$ in view of Remark~\ref{basic}. So, by the universal property of the above pullback diagram, with the roles of $X^m$, $\mathcal{D}_{m}$, and $\sigma$ replaced, respectively, with $C_i$, $\mathcal{D}_{m,i}$, and $\sigma_i$ for $i = 1,2$, we conclude that there exists a continuous section $\tau_i : C_i \to \mathcal{D}_{m,i}$ of the fibration $\sigma_i: \mathcal{D}_{m,i} \to C_i$. 
\end{proof}

For a fixed $m \ge 2$, let $\Delta_2:SP^{2}(X) \to (SP^{2}(X))^m$ be the diagonal map. Then, we use some ideas from~\cite[Theorem 4.4]{Ja} to prove the following theorem in Alexander--Spanier cohomology.

\begin{theorem}\label{boundtcm}
Suppose that $\alpha_i^* \in H^{k_i}((SP^{2}(X))^m; R)$ for ring $R$ and $k_i \ge 1$ are cohomology classes  for $i = 1,2$ such that $\Delta_2^*(\alpha_i^*) = 0$. If $\alpha_i$ are their images under the induced homomorphism $(\pa_2^m)^*$ such that $\alpha_1 \smile  \alpha_2 \neq 0$, then 
\[
\iTC_m(X) \ge 2.
\]
\end{theorem}

\begin{proof}
Let us define $g: SP^{2}(X) \to P(SP^2(X))$ by $g(x+y)=c_{x+y}$ for each $x+y\in SP^2(X)$, where $c_{x+y}(t)=x+y$ for all $t\in I$. Then, $g$ is a homotopy equivalence.
Now, suppose $\iTC_m(X) < 2$. Then by Lemma~\ref{second}, there exists a cover $\{C_1,C_2\}$ of $X^m$ and maps $\tau_i:C_i \to \mathcal{D}_{m,i}$ such that $\sigma_i \tau_i = \mathbbm{1}_{C_{i}}$, see Diagram~\ref{four4}. Let us fix some $i\in\{1,2\}$ and consider the following commutative diagram.
    \begin{equation}\label{five5}
    \begin{tikzcd}[every arrow/.append style={shift left}]
H^{k_i}(C_i;R)  \arrow{r}{\sigma_i^*}
& 
H^{k_i}(\mathcal{D}_{m,i};R) \arrow{l}{{\tau_i^* \vphantom{1}}}
\\
H^{k_i}((SP^{2}(X))^m;R) \arrow{r}{\zeta^*} \arrow[swap]{u}{(\pa_2^m)^*} \arrow{d}{\Delta_2^*}
& 
H^{k_i}(P(SP^{2}(X));R) \arrow[swap]{u}{a^*} \arrow{dl}{g*}
\\
H^{k_i}(SP^2(X);R) 
& 
\end{tikzcd}
    \end{equation}
    Since $g^*$ is an isomorphism, $\Delta_2^*(\alpha_i^*) = 0 \implies \zeta^*(\alpha_i^*) = 0$ in the triangle in the bottom. The commutativity of the square on the top then gives
    $$
    (\pa_2^m)^*(\alpha_i^*) = \tau_i^* a^*\zeta^*(\alpha_i^*) = 0.
    $$
Due to the long exact sequence of the pair $((SP^{2}(X))^m,C_i)$ in Alexander--Spanier cohomology, 
$$
\cdots \to H^{k_i}((SP^{2}(X))^m,C_i; R) \xrightarrow{j_i^*} H^{k_i}((SP^{2}(X))^m;R) \xrightarrow{(\pa_2^m)^*} H^{k_i}(C_i;R) \to \cdots,
$$
there exists $\ov{\alpha_i^*} \in H^{k_i}((SP^{2}(X))^m,C_i; R)$ such that $j_i^*(\ov{\alpha_i^*}) = \alpha_i^*$. Further, let $(\pa_2^m)^*(\ov{\alpha_i^*}) = \ov{\alpha}_i \in H^{k_i}(X^m,C_i; R)$. Consider the following commutative diagram, where $k=k_1+k_2$ and $j$ (resp. $j'$) is the sum of the maps $j_1$ and $j_2$ (resp. $j_1'$ and $j_2'$).
\begin{equation}\label{six6}
\begin{tikzcd}[contains/.style = {draw=none,"\in" description,sloped}]
H^{k}(X^m;R)  
& 
\arrow[swap]{l}{(j')^*} H^{k}\left(X^m,C_1 \cup C_2;R\right)  
\\
H^{k}\left((SP^{2}(X))^m;R\right)  \arrow[swap]{u}{(\pa_2^m)^{*}}
& 
H^k \left((SP^{2}(X))^m,C_1 \cup C_2;R\right) \arrow[swap]{u}{(\pa_2^m)^{*}} \arrow[swap]{l}{j^*} 
\end{tikzcd}
\end{equation}
The cup product $\ov{\alpha_1^*} \smile \ov{\alpha_2^*}$ in the bottom-right goes to $\alpha_1 \smile \alpha_2 \neq 0$ in the top-left. But in the process, it factors through $\ov{\alpha}_1 \smile  \ov{\alpha}_2 \in H^k (X^m,X^m;R) = 0$. This is a contradiction. Hence, we must have $\iTC_m(X) \ge 2$.
\end{proof}

As explained at the end of Section~\ref{prepo}, this lower bound cannot be improved in general using our methods. However, if Lemma~\ref{second} is proved for some $n \ge 3$ by some other method, then our proof of Theorem~\ref{boundtcm} will still work by accordingly replacing the role of the two cohomology classes $\alpha_i^*$ with more number of classes.

\section{Computations}\label{Computations}

In this section, we use our lower bounds from Section~\ref{Bounds} to justify the non-triviality of the intertwining invariants in several cases.

\subsection{For $\icat$}\label{last1}
Computations of $\dcat(X)$ for various spaces $X$ were done in~\cite[Section 6.1]{DJ} using their rational cup lengths $c\ell_{\Q}(X)$. Here, we will use Corollary~\ref{rationalcl} to get various examples of closed manifolds $X$ for which $\icat(X) = 2$. 

\begin{example}\label{2cat}
Using rational and real cup lengths and the $\cat$ values, we get:
\vspace{-1mm}
\begin{itemize}
        \itemsep-0.25em 
    \item $\icat(\Sigma_g) = 2$ for each closed orientable surface $\Sigma_g$ of genus $g \ge 1$.
    \item $\icat(\C P^2) = 2$.
    \item $\icat( S^{k_1}\times S^{k_2}) = 2$ for each $k_i \ge 1$.
    \item If $(M^4,\omega)$ is a closed simply connected $4$-dimensional symplectic manifold such that $[\omega]^2 \ne 0$ for the de Rham cohomology class $[\omega] \in H^2(M^4;\R)$ corresponding to the 
    symplectic $2$-form $\omega$, then $\icat(M) = 2$. 
\end{itemize}
\end{example}

\begin{remark}\label{sum}
    Let $M$ and $N$ be two closed $k$-dimensional manifolds of which at least one is orientable.
    We know that $\wt{H}_i(M \# N) = \wt{H}_i(M) \oplus \wt{H}_i(N)$ for all $i < k$. Let $(m,n) \in H^i(M \#N)$ and $(m',n') \in H^j(M \# N)$. Then,
    $$
    (m,n) \smile (m',n') = \begin{cases}
        \left(m\smile m',n \smile n'\right) & : i + j < k \\
        m \smile m' + n \smile n' & : i + j = k
    \end{cases}
    $$
    Thus, for $\F \in \{\Q,\R\}$, $\max\{c\ell_{\F}(M),c\ell_{\F}(N) \}\le c\ell_{\F}(M \# N)$.
\end{remark}

\begin{example}\label{sumuse}
    Using Corollary~\ref{rationalcl} with the corresponding $\F$ cup lengths, Remark~\ref{sum}, and the connected sum formula $\cat(M \# N) = \max\{\cat(M),\cat(N)\}$ from~\cite[Proposition 11]{DS}, we get the following conclusions.
    \begin{itemize}
            \itemsep-0.25em 
    \item For $3\le n \le 5$, let $K^n$ denote a closed simply connected $n$-dimensional manifold and $M^4$ be an orientable symplectic manifold as in Example~\ref{2cat}. Note that $\cat(K^n)\le 2$ due to~\cite[Theorem 1.50]{CLOT}. Thus, the $\icat$ value for each of the following five classes of manifolds is $2$.
    \[
    K^3\# (S^1 \times S^2), \hspace{1mm} K^4 \# \hspace{0.5mm}\C P^2, \hspace{1mm} K^4 \# M^4, \hspace{1mm}  K^4 \# (S^2 \times S^2), \hspace{1mm} K^5 \# (S^3 \times S^2).
    \]
     \item $\icat(\Sigma_g\# L^2) = 2$ for any closed $2$-dimensional manifold $L^2$ and $g \ge 1$.
    \end{itemize}
\end{example}

\begin{remark}
    Let $N_g$ denote the closed non-orientable surface of genus $g \ge 1$. For each $g \ge 2$, $\dcat(N_g)=2$ is computed using the covering map inequality from~\cite[Theorem 3.7]{DJ}. Since we do not know whether such an inequality holds in the case of $\icat$ as well, we instead use the fact that that $N_g=\Sigma_1 \# N_{g-2}$ for $g \ge 3$. So, we can take $L^2 = N_{g-2}$ in Example~\ref{sumuse} and get $\icat(N_g)=2$ for all $g \ge 3$.
\end{remark}

Since $\icat(N_1) = 1$, this leaves open only the case of $N_2$, the Klein bottle. 

\begin{example}
Let $L$ be a smooth simply connected manifold of dimension $4$, other than $S^4$. Then $L$ is homotopic to $M \# N$ where $M \in \{\C P^2,S^2 \times S^2,-\C P^2\}$. By Remark~\ref{sum} and the fact that $\cat(L) = 2$ (see~\cite[Section 7]{DS}), we get $\icat(L) = 2$. 
\end{example}

\subsection{For $\iTC_m$} For the case of $\iTC:=\iTC_2$, we will use Corollary~\ref{rationalzcl} and~\cite{Far1} to get some examples of locally finite CW complexes $X$ for which $\iTC(X) = 2$. 

\begin{example}Using rational zero-divisor cup lengths and $\TC$ values, we get:
\vspace{-1.5mm}
\begin{itemize} 
        \itemsep-0.25em 
    \item $\iTC(S^{2n-1}) = 1$ and $\iTC(S^{2n}) = 2$ for each $n\ge 1$.
    \item $\iTC(S^{2k_1-1}\times S^{2k_2-1}) = 2$ for any $k_i \ge 1$.
    \item $\iTC(X) = 2$ if $X$ is a finite graph having first Betti number greater than $1$.
\end{itemize}
For the $2$-torus $\Sigma_1$, since it is a topological group, we get from Corollary~\ref{usetopgrp} and Example~\ref{2cat} that $\iTC(\Sigma_1) = \icat(\Sigma_1) = 2$.
\end{example}

For all the closed manifolds $X$ considered in the previous subsection, we see that $\icat(X) = \dcat(X) = \cat(X)$. Also, for all the finite CW complexes $Y$ considered till this point in this section, we see that $\iTC(Y) = \dTC(Y) = \TC(Y)$.

We now focus on computing $\iTC_m$ values for certain finite CW complexes for each $m \ge 3$. The technique is similar to that used in~\cite[Section 7]{Ja} for $\dTC_m$. However, the proofs need to be modified because we need cohomology classes in the ring of $(SP^{2}(X))^m$ now as opposed to cohomology classes in the ring of $SP^{2}(X^m)$ in the case of $\dTC_m$ for $n=2$. We prove the following statement to explicitly demonstrate our technique, using which, similar computations can be done to get more estimates for $\iTC_m(X)$ in the same way as the computations were done in~\cite[Section 7.B]{Ja} to get more estimates for $\dTC_m(X)$. 

\begin{proposition}\label{tech}
    Let $X$ be a finite CW complex and $m \ge 3$. If $H^d(X;\F) \ne 0$ for some $d \ge 1$ and $\F \in \{\Q,\R\}$, then $\iTC_m(X) \ge 2$.
\end{proposition}

\begin{proof}
    For brevity, call $Y = SP^{2}(X)$. Let $v \in H^d(X;\F)$ such that $v \neq 0$. Due to Proposition~\ref{useful}, there exists $w \in H^d(Y;\F)$ such that $\pa_{2}^*(w) = v$, where $\pa_{2}:X \to Y$ is the diagonal map. For $1 \le i \le m$, let $\text{proj}_i : X^m \to X$ and $k_i: Y^m\to Y$ be the projections onto the respective $i^{th}$ coordinates. Let $\pa_2^m:X^m \to Y^m$ be the product map, and $\Delta: X \to X^m$ and $\Delta_2:Y \to Y^m$ be the diagonal maps. Then for each $i$, the following diagram commutes.
\begin{equation}\label{seven7}
\begin{tikzcd}
Y
&
Y^m \arrow{l}[swap]{k_i}
&
Y \arrow{l}[swap]{\Delta_{2}}
\\ 
X \arrow{u}{\pa_{2}}
& 
X^m \arrow{l}{\text{proj}_i} \arrow{u}{\pa_{2}^{m}}
&
X \arrow{l}{\Delta} \arrow{u}[swap]{\pa_{2}}
\end{tikzcd}
\end{equation}
For each $1 \le i \le m-1$, we define maps $\phi_i:H^d(X;\F) \oplus H^d(X;\F) \to H^d(X^m;\F)$ as $\phi_i(x \oplus y) = \proj_i^*(x) - \proj_m^*(y)$ and elements $\alpha_i = \phi_i(v \oplus v)$ exactly like in~\cite[Proposition 7.1]{Ja}. Similarly, we define $\psi_i:H^d(Y;\F) \oplus H^d(Y;\F) \to H^d(Y^m;\F)$ as $\psi_i(a \oplus b) = k_i^*(a) - k_m^*(b)$ and $\alpha_i^* = \psi_i(w \oplus w)$. Then the following commutes.
\begin{equation}\label{eight8}
\begin{tikzcd}[contains/.style = {draw=none,"\in" description,sloped}]
w \oplus w \ar[r,mapsto] \ar[d,contains] & \alpha_i^* \ar[d,contains] \ar[r,mapsto] & w-w = 0 \ar[d,contains]
\\
H^{d}(Y;\F)  \oplus H^{d}(Y;\F)  \arrow{r}{\psi_i}  \arrow[swap]{d}{\pa_{2}^* \oplus \hspace{1mm} \pa_{2}^* \hspace{1mm}}
&
H^d(Y^m;\F)  \arrow{r}{\Delta_{2}^*} \arrow[swap]{d}{(\pa_{2}^{m})^*}
&
H^d(Y;\F) \arrow[swap]{d}{\pa_{2}^*}
\\
H^{d}(X;\F)  \oplus H^{d}(X;\F)  \arrow{r}[swap]{\phi_i} 
&
H^d(X^m;\F)  \arrow{r}[swap]{\Delta^*} 
&
H^d(X;\F)
\\
v \oplus v \ar[r,mapsto] \ar[u,contains] & \alpha_i \ar[u,contains] \ar[r,mapsto] & v-v = 0 \ar[u,contains]
\end{tikzcd}
\end{equation}
On the top-right, $w-w =0$ because the top (resp. bottom) row when restricted to either of the components $H^d(Y;\F)$ (resp. $H^d(X;\F)$) is an isomorphism. Due to~\cite[Proposition 3.5]{Rud}, we have 
$$
\alpha_1 \smile \cdots \smile \alpha_{m-1} \neq 0.
$$
Since $m \ge 3$, we can take $k_i=d$ and $\alpha_i^*\in H^{k_i}((SP^2(X))^m;\F)$ for $i = 1,2$ in the statement of Theorem~\ref{boundtcm} to get $\iTC_m(X) \ge 2$.
\end{proof}

Of course, if $H^d(X;\F) \ne 0$ for some $d \ge 1$ and $\F \in \{\Q,\R\}$, then $\iTC_3(X) \ge 2$ by Proposition~\ref{tech}, and using Proposition~\ref{ineq6}, we get that $\iTC_m(X) \ge 2$ for all $m \ge 4$ as well. However, we proved Proposition~\ref{tech} by specifying the indexes $m$ so that if the lower bound in Theorem~\ref{boundtcm} is improved, then the technique of Proposition~\ref{tech} remains useful more generally to obtain for each $m \ge 3$ better lower bounds to $\iTC_m(X)$ that depend on $m$, just like in~\cite[Proposition 7.1]{Ja} in the case of $\dTC_m$.

\begin{remark}
We note that the hypothesis in 
Proposition~\ref{tech}, that $H^d(X;\F) \ne 0$ for some $d \ge 1$ and $\F\in\{\Q,\R\}$, cannot be dropped for any $m \ge 3$ due to the example of Higman's group $\H$ for which $\iTC_m(\H) = 1$ for all $m \ge 2$ by Proposition~\ref{higman}.
\end{remark}

\begin{corollary}\label{veryobv}
    If $X$ is a closed orientable manifold or a closed non-orientable surface of genus $g > 1$, then $\iTC_m(X) \ge 2$ for all $m \ge 3$. 
\end{corollary}

\begin{example}
    For each $n \ge 1$, $\iTC_3(S^{2n-1}) = 2$ due to Corollary~\ref{veryobv} and the fact that $\TC_3(S^{2n-1}) = 2$, see~\cite[Section 4]{Rud}. Hence, the non-decreasing sequence $\{\iTC_m(S^{2n-1})\}_{m \ge 2}$ is not constant for any $n$. It is noteworthy that the equality $\iTC_{3}(S^{2n-1}) = 2$ cannot be obtained by any other result mentioned in this paper. So, Corollary~\ref{veryobv} indeed gives us new computations.
\end{example}

\section*{Acknowledgement}
The author would like to thank Alexander Dranishnikov for his kind guidance and various insightful discussions. The author is also grateful to the anonymous referee for several helpful suggestions and corrections to the first draft and the first revision of this paper.


\end{document}